\documentclass[a4paper,11pt]{article}
\usepackage[
a4paper,
top=1in,
bottom=1in,
left=1in,
right=1in]{geometry}
\usepackage[scaled=.92]{helvet}

\usepackage{hyperref}

\usepackage{graphicx} % Required for inserting images
\usepackage{ytableau}
\usepackage{mathrsfs}
\usepackage{tikz-cd}
\usepackage{amsmath}
\usepackage{quiver}
\usepackage[dvipsnames]{xcolor}
\usepackage{amsthm}
\usepackage{amssymb}
\usepackage{comment}
\usepackage{float}
\renewcommand\labelenumi{\upshape(\roman{enumi})}
\renewcommand\theenumi\labelenumi
\newcommand{\C}{\mathbb{C}}

\newcommand{\N}{\mathbb{N}}

\newcommand{\R}{\mathbb{R}}
\newcommand{\Z}{\mathbb{Z}}

\newcommand{\Horn}{\mathrm{Horn}}
\newcommand{\Honey}{\mathrm{Honey}}
\newcommand{\Frame}{\mathrm{Frame}}

\newtheorem{theorem}{Theorem}[section]
\newtheorem{lemma}[theorem]{Lemma}
\newtheorem{corollary}[theorem]{Corollary}
\newtheorem{proposition}[theorem]{Proposition}

\theoremstyle{definition} %makes text non-italicised

\newtheorem{remark}[theorem]{Remark}
\newtheorem{definition}[theorem]{Definition}
\newtheorem{claim}[theorem]{Claim}

\newtheorem{fact}[theorem]{Fact}

\newcommand{\Dora}[1]{{\color{red} \sf $\clubsuit$ Dora: [#1]}}

\title{Honeycombs and Sums of Hermitian Matrices, Revisited}
\author{Ankur Moitra\thanks{\texttt{moitra@mit.edu}. Supported in part by DARPA expMath, NSF-CCF 2430381, an ONR grant, and a David and Lucile Packard Fellowship.} \and Alexander Postnikov\thanks{\texttt{apost@math.mit.edu}.  Supported in part by NSF grant DMS-2054129.} \and Dora Woodruff\thanks{\texttt{dorawood@mit.edu}. Supported in part by NSF GRFP 2141064.}}
\date{\today}

\begin{document}

\maketitle

\begin{abstract}
We give a new proof of the celebrated theorem of Knutson and Tao that the spectra of triples $A, B, A+B$ of Hermitian matrices exactly correspond to positions
of boundary rays of honeycombs. Most importantly, our proof gives new insights into why honeycombs are related to Hermitian matrices in the first place. Our proof is axiomatic: We distill four essential properties shared by honeycombs and spectra of Hermitian triples, and show that any two objects sharing these four properties must be equivalent.  
%We define related objects called frameworks also satisfying the four axioms, which are reminiscent of Feynman diagrams. 
In this way, we argue that honeycombs are `model organisms' for Hermitian triples \---- they are families of objects satisfying the same defining properties, but in more obvious ways. 
\end{abstract}

%\newpage

\section{Introduction}
Given an $n \times n$ Hermitian matrix $A$, we can sort its eigenvalues into a nonincreasing $n$-tuple $\lambda \in \R^n$. In $1912$, Hermann Weyl asked the following fundamental question: Given the eigenvalues of Hermitian matrices $A$ and $B$, can we determine all possible eigenvalues of $A+B$? We refer to triples $(\lambda, \mu, \nu)$ of nonincreasing $n$-tuples as \textit{Horn triples} if there exist Hermitian matrices $A, B$ and $C$ with spectra $\lambda, \mu$ and  $\nu$ respectively satisfying $A+B=C$. 

Many famous inequalities in matrix analysis can be thought of as necessary conditions for Horn triples. For example, Weyl's inequality is the bound $\lambda_1 + \mu_1\geq \nu_1$ and more generally
 $$\lambda_i + \mu_j \geq \nu_{i+j-1}$$
for any $i + j -1 \leq n$. Similarly, the Ky Fan inequalities \cite{fan} are the bound
 $$\sum_{j =1}^r \lambda_j + \sum_{k = 1}^r \mu_k\geq \sum_{i=1}^r \nu_i$$
for any $r \leq n$. But what are the sufficient conditions? 
In $1962$, Alfred Horn \cite{horn} formulated a conjectural master answer to Weyl's question, describing such triples $(\lambda, \mu, \nu)$ as points inside a certain polyhedral cone. Horn's Conjecture waited $36$ years for a resolution. A major step was taken by Klyachko \cite{klyachko} and the final step was completed by Knutson and Tao when they proved the Saturation Conjecture in \cite{honeycombs1}. An important component of Knutson and Tao's solution was their theory of \textit{honeycombs}. Honeycombs, which are graphical representations of Berenstein-Zelevinsky patterns \cite{BZ},
are certain configurations of line segments inside the plane.
%$\{(x,y,z) \in \R^3 \mid x+y+z=0\}$. 
Knutson and Tao proved the following fundamental theorem:

\begin{theorem}[Knutson-Tao, \cite{honeycombs1}]\label{thm:main_thm}
    Let $\lambda, \mu$ and  $\nu$ be nonincreasing $n$-tuples of real numbers. Then there exist Hermitian matrices $A, B$  and $A+B$ with spectra $\lambda, \mu$ and $ \nu$ if and only if there exists a honeycomb with positions of boundary rays given by $\lambda, \mu$ and $ -\nu$. 
\end{theorem}

The original proof of Theorem \ref{thm:main_thm} connects the representation theory of $GL_n(\mathbb{C})$, Littlewood-Richardson coefficients, geometric invariant theory and symplectic geometry. Since then, alternative proofs have appeared. For example, Derksen and Weyman \cite{quivers} gave a proof of the Saturation Theorem through semi-invariants of quivers. Belkale \cite{belkale} gave a geometric proof via Schubert calculus and intersection theory. The most distinct approach from Klyachko, Knutson and Tao's work is Speyer's.  In \cite{speyer}, Speyer replaced all representation theory with the theory of Vinnikov curves and Viro's patchworking method. 

In this article, we give a new proof of Theorem \ref{thm:main_thm}. Our goal is to give a more direct, elementary approach, going straight from Hermitian matrices to honeycombs and vice versa without using any algebraic geometry, Schubert calculus, or representation theory as a middleman. After all, the statement of Theorem \ref{thm:main_thm} is explainable to an undergraduate linear algebra student \---- but `why honeycombs?' is a harder question to answer. 

Our approach is axiomatic. We distill four properties, which we call \textit{base case, direct sum, convexity, and splitting}, and show that any two families of objects satisfying these four properties must be equivalent. 
We define \textit{frameworks}, which are objects related to honeycombs, although not equivalent to honeycombs.
Then we show that the three different families of objects, namely, Horn triples, honeycombs, and frameworks,
 satisfy the axioms. The proofs that honeycombs and frameworks satisfy the axioms are combinatorial and straightforward and only rely on having the right definitions. The proofs of these axioms for Horn triples, on the other hand, are more involved. In this way, we like to think of honeycombs
 and frameworks
 as \textit{model organisms} for Horn triples in that they are objects satisfying the same defining properties, but in more obvious ways. Thus honeycombs can 
 be a sandbox for exploring more fine-grained structural properties of Horn triples and Littlewood-Richardson coefficients. 

%We can think of this as a \textit{Black Box Problem}.   We have a black box, or a certain computational model that, for a triple $(\lambda,\mu,\nu)$, outputs
%`True' or `False'.   We do not know the internal mechanics of its decision-making process, but we know that, from general principles, the four axioms 
%should hold.   Then we present $3$ different models for the black box: Hermitian matrices, honeycombs, and frameworks.   
%This
%perspective is analogous to \textit{quantum physics}. 
%$\lambda,\mu,\nu$ represent measurements of momenta of scattering particles. 
%From a physics point of view, $\lambda,\mu,\nu$ represent measurements of momenta of a system of scattering particles,
%frameworks is a version of \textit{Feynman diagrams} and, for example, the conservation axiom is the 
%\textit{momentum conservation law.}

In Section \ref{sec:honeycombs_and_frameworks}, we recall two definitions of honeycombs, and define frameworks.
%related objects, which we call \textit{frameworks}.   
To illustrate the difference between the notions of honeycombs and frameworks, we use frameworks to simplify the proof of the Saturation Conjecture.
In Section \ref{sec:axioms}, we state our four axioms and show that any two families of objects satisfying the axioms must be equal. In Section \ref{sec:honeycomb_axioms}, we prove that our axioms hold for honeycombs and frameworks. Three of the axioms are immediate from the definitions; the fourth follows from the concept of frameworks. In Section~\ref{sec:Horn_triples}, we prove that Horn triples also satisfy our axioms, completing the proof of Theorem~\ref{thm:main_thm}. 

\section{The Horn problem}\label{subsec:horn}

Let us briefly summarize the history of the Horn Problem. For a more complete account, we direct the reader to Fulton's excellent survey \cite{fulton}. After Weyl posed the problem of characterizing the eigenvalues of sums of Hermitian matrices, necessary inequalities for Horn triples $(\lambda, \mu, \nu)$ were collected in the works of Lidskii, Weilandt, Thompson, Freed (see \cite{lidskii}, \cite{weilandt}, \cite{thompson}), and others. All of these inequalities that were discovered took the form
\begin{equation}\label{eq:lambda_mu_nu_inequality}
\sum_{i \in I} \lambda_i + \sum_{j \in J} \mu_j \geq \sum_{k \in K} \nu_k,
\end{equation}
where $I, J, K$ are certain subsets of $[n]:=\{1,\dots,n\}$ of the same cardinality. Two questions followed from this observation: 
\begin{enumerate}
    \item For which triples $(I, J, K)$ does this expression 
    %\eqref{eq:lambda_mu_nu_inequality} 
    yield a necessary inequality?
    \item Are these inequalities
    %, together with the obvious trace condition $\sum_i \lambda_i + \sum_i \mu_i = \sum_i \nu_i$
    %and the inequalities $\lambda_1\geq \cdots\geq \lambda_n$, $\mu_1 \geq \cdots\geq \mu_n$, and $\nu_1\geq\cdots\geq \nu_n$, 
    sufficient to characterize all Horn triples $(\lambda, \mu, \nu)$?
\end{enumerate}

Horn~\cite{horn} conjectured that the answer to the second question was positive, and conjecturally described all such triples $(I, J, K)$ via the following explicit 
recursive process.   
%Let $\mathrm\Horn_n\subset\R^{3n}$ denote the set of all Horn triples $(\lambda,\mu,\nu)$. 
For a subset $I=\{i_1<i_2<\cdots<i_k\}\subset[n]$, let $\pi_I$ be the partition $\pi_I := (i_k-k,\dots,i_2-2, i_1-1)$.

\begin{theorem}[Horn's Conjecture]  A triple $(\lambda,\mu,\nu)$ of $n$-tuples is a Horn triple if and only if 
$\sum \lambda_i + \sum \mu_i = \sum \nu_i$,
$\lambda_1\geq \cdots\geq \lambda_n$, $\mu_1 \geq \cdots\geq \mu_n$, $\nu_1\geq\cdots\geq \nu_n$, and inequality~\eqref{eq:lambda_mu_nu_inequality} holds for any triple $(I,J,K)$ of proper nonempty subsets of $[n]$ of the same cardinality $k=|I|=|J|=|K|<n$ such that $(\pi_I,\pi_J,\pi_K)$
is a Horn triple.
\end{theorem}

\noindent Horn proved his conjecture for $n= 3, 4$.  But the recursion quickly becomes complicated. For example when $n= 7$, Horn's recursive algorithm yields $2062$ inequalities (not necessarily all independent). 

Work of Thompson, Totaro, Fulton, Klyachko (\cite{totaro}, \cite{klyachko}) and others connected the Horn problem to Schubert calculus, as well as the representation theory of $GL_n(\C)$. Klyachko \cite{klyachko} proved Horn's Conjecture modulo one step: the Saturation Conjecture for the Littlewood-Richardson coefficients. Knutson and Tao's theory of honeycombs completed this last step, giving a full resolution to the Horn Conjecture. 
%Honeycombs are themselves based on the work of
%Berenstein and Zelevinsky \cite{BZ} and the notion of \textit{Berenstein-Zelevinsky patterns.}

Since Knutson-Tao's resolution of Horn's Conjecture, honeycombs and the related \textit{puzzles} and \textit{hives} from \cite{honeycombs2} have proven to be extremely useful models for understanding the Littlewood-Richardson coefficients, Schubert calculus, and spectral theory of Hermitian matrices. For instance, honeycombs lead to a proof that deciding the positivity of the Littlewood-Richardson coefficients can be solved in polynomial time \cite{algorithmic1} \cite{algorithmic2} (see \cite{algorithmic3} for further relationships with complexity theory). Puzzles have been used to study the cohomology of the $2$-step flag variety following conjectures of Knutson \cite{puzzles-2}, and hive-inspired objects called skeps were recently used by Speyer \cite{skeps} to solve a conjecture of Lam, Postnikov and Pylyavskyy \cite{log-concave} on Schur positivity. 

\begin{remark}
The works \cite{honeycombs1} and \cite{honeycombs2} provided an `elementary' combinatorial proof that 
the set of possible positions of boundary rays of honeycombs is described by the same recursive procedure as in Horn's Conjecture.
%implies an `elementary' combinatorial proof of Horn's conjecture.  
The argument uses an interplay between the notions of honeycombs and puzzles.
In a nutshell, it goes as follows:  The inequalities describing possible positions $(\lambda,\mu,\nu)$ of boundary rays of honeycombs of size $n$ correspond to combinatorial types of pairs of overlayed honeycombs $H_1$ and $H_2$ such that all line segments of $H_1$ intersect with line segments of $H_2$ at $120^\circ$ 
angles as we go clockwise from $H_1$ to $H_2$.   Such pairs can be identified with puzzles.   On the other hand, puzzles are in bijection with certain
\textit{integer} honeycombs of sizes  $< n$.   Finally, all possible integer positions of boundary rays of honeycombs can be represented by integer honeycombs
(the Saturation Conjecture.)
%This, together with the Saturation Conjecture (proved in \cite{honeycombs1} using a combinatorial argument),

Thus an `elementary' proof of Theorem~\ref{thm:main_thm} would provide an `elementary' solution of the Horn problem.
\end{remark}

\begin{comment}
Independently from \cite{honeycombs1},  \textit{web diagrams} and \textit{web functions}, which are closely related to honeycombs, were introduced in \cite{webs}.
Combinatorial constructions from \cite{honeycombs1} and \cite{webs} are based on the earlier work by Berenstein and Zelevinsky \cite{BZ} and the notions of 
\textit{Berenstein-Zelevinsky patterns} and \textit{Berenstein-Zelevinsky polytopes.}  
The work \cite{BZ} was one of the early predecessors of the theory of \textit{cluster algebras} [REFERENCE] by Berenstein, Fomin, and Zelevinsky.
Web diagrams lead to the introduction of \textit{plabic graphs} that play a key role in the study of combinatorics of the positive Grassmannain
[REREFENCE].    Plabic graphs appeared in many other areas of mathematics and physics, including statistical physics, electrical networks, 
theory of solitons, etc.  They turned out to be intimately related to the study 
of \textit{scattering amplitudes} in $\mathcal{N}=4$ supersymmetric Yang-Mills theory, \textit{Grassmannian formula} 
for the scattering amplitudes, and the study of the \textit{amplituhedron},
see [REFERENCE to Book}] and [REFERENCE TO NIMA, TRNKA].
In physics, plabic graphs appear as \textit{on-shell diagrams}.

%We hope that the ideas and constructions that we discuss in this paper, might also turn out to be helpful to link many different areas of 
%mathematics and physics.
\end{comment}

\section{Notation and conventions}\label{sec:notation}

The \textit{fundamental Weyl chamber} (of type A) is $\mathfrak{C}_n:=\{ \lambda \in \R^n \mid \lambda_1 \geq \lambda_2 \dots \geq \lambda_n\}$. 
Points $\lambda\in\mathfrak{C}_n$ 
with positive integer entries $\lambda_i$ are exactly integer partitions with $n$ parts.
In general, we can view points $\lambda\in \mathfrak{C}_n$ as ``$\R$-valued partitions''; and refer to them as simply `partitions'.

%Points of $\mathfrak{C}_n$ are given by partitions; points on the interior of $\mathfrak{C}_n$ are given by strictly decreasing lists of reals. This distinction between boundary points and interior points of $\mathfrak{C}_n$ will be important for us. Note that $\Horn_n(\lambda, \mu)$ is a subset of $\mathfrak{C}_n$. 

%For us, a \textit{partition} $\lambda$ is a weakly decreasing list of \textit{reals} $\lambda_1 \geq \lambda_2 \dots \geq \lambda_n$. Usually, a partition is a nonincreasing list of \textit{nonnegative integers}. We will abuse notation and allow $\lambda$ to have parts that are noninteger and negative. The \textit{length} of $\lambda$ is $n$. The \textit{parts} of $\lambda$ are the entries $\lambda_i$ of the list. We denote the set of length $n$ partitions by $\text{Par}_n$. 

%If $\lambda=(\lambda_1 \dots \lambda_n)\in \mathfrak{C}_n$ and $I = \{i_1, i_2 \dots i_m\} \subseteq [n]$, we use $\lambda_I$ to denote the partition $(\lambda_{i_1}, \lambda_{i_2} \dots \lambda_{i_m})$. The complement of $I$ in $[n]$ is denoted $I^c$, and $\lambda_{I^c}$ is defined analogously. For example, if $\lambda = (5,3,2,2)$, $I = \{1,3\}$, then $\lambda_I = (5,2)$ and $\lambda_{I^c} = (3,2)$. 

Consider two operations on partitions, which we call `sum' and `direct sum', and denote by the symbols `$+$' and `$\oplus$', respectively.
For two partition $\lambda,\mu\in\mathfrak{C}_n$, the \textit{sum} of $\lambda$ and $\mu$ is $\lambda+\mu:=(\lambda_1+\mu_1,\dots,\lambda_n+\mu_n)$.
For two partitions $\lambda\in\mathfrak{C}_m$ and $\mu\in\mathfrak{C}_n$,
the \textit{direct sum} $\lambda \oplus \mu\in\mathfrak{C}_{m+n}$ is given by concatenating $\lambda$ and $\mu$, then sorting entries so that they are weakly decreasing. For example, $(5,4,2)+(4,3,1)=(9,7,3)$ and 
$(5,4,2) \oplus (4, 3,1) = (5,4,4,3,2,1)$. The rationale for calling this operation a `direct sum' will be apparent later. 

%There is another way of `adding' two partitions: namely, $\lambda + \mu = (\lambda_1 + \mu_1, \lambda_2 + \mu_2 \dots \lambda_n + \mu_n)$. We denote this addition with the $+$ symbol. 

As in the introduction, we will call a triple $(\lambda, \mu, \nu)$ 
of spectra a \textit{Horn triple} if there exist Hermitian matrices $A, B, A+B$ with spectra $\lambda, \mu, \nu$. 

\begin{definition}\label{def:Horn_n}
    For $\lambda, \mu\in\mathfrak{C}_n$,  $\Horn_n(\lambda, \mu)$ is the set of all $\nu \in \mathfrak{C}_n$ such that $(\lambda, \mu, \nu)$ is a Horn triple. 
\end{definition}

Consider the plane $\R^3_{\Sigma = 0} := \{(x, y, z) \in \R^3 \mid x+y+z=0\}$. Inside the plane $\R^3_{\Sigma=0}$ we identify six distinguished vectors, which we associate to compass directions: $(1, -1, 0)$ is North, $(1, 0, -1)$ is Northeast, $(0, 1, -1)$ is Southeast, $(-1, 1, 0)$ is South, $(-1, 0, 1)$ is Southwest, and $(0, -1, 1)$ is Northwest, see Figure~\ref{fig:compass}. These directions are at $60^\circ$ angles to each other, rather than being perpendicular. We use this compass terminology nonetheless. 

%\textit{Honeycombs} will always be drawn inside the plane $\R^3_{\Sigma = 0}$.

Honeycombs are certain configurations of line segments and rays drawn inside the plane $\R^3_{\Sigma = 0}$ such that all line segments and rays 
are parallel to the distinguished compass directions.
We will see that the condition $x+y+z=0$ can be thought of as imposing the trace condition 
%$\mathrm{tr}(A+B)=\mathrm{tr}(A)+\mathrm{tr}(B)$ 
for Horn triples. 

\begin{figure}
\begin{center}
    \includegraphics[scale=0.75]{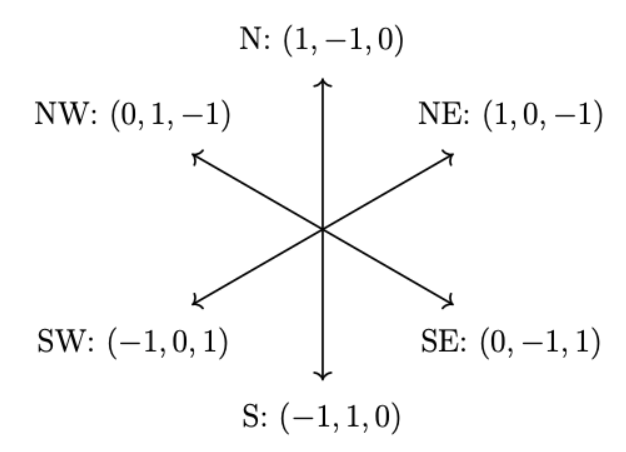}
\end{center}
\caption{Compass directions in $\R^3_{\Sigma = 0}$}
\label{fig:compass}
\end{figure}

\section{Honeycombs and frameworks}\label{sec:honeycombs_and_frameworks}

In this section, we give two cryptomorphic definitions of honeycombs. Honeycombs were introduced by Knutson and Tao in \cite{honeycombs1}.  
Independently, closely related objects called \textit{web diagrams} and \textit{web functions} were defined in \cite{webs}.  These works rely on the prior work 
by Berenstein and Zelevinsky \cite{BZ}.   Basically, honeycombs are graphical representations of \textit{Berenstein-Zelevinsky patterns.}

Then, we define \textit{frameworks}, certain embedded graphs similar to honeycombs; however, there are important differences between frameworks and honeycombs. In particular, there is at least one framework associated to every honeycomb, but this correspondence is not one-to-one. 
%In order to illustrate this distinction, we use frameworks to give a concise proof of the Saturation Conjecture.
Frameworks will prove useful in Section~\ref{sec:honeycomb_axioms}.

\subsection{Honeycombs as graph embeddings}

In this paper, the term `graph' means a finite simple graph $G=(V,E)$ where we allow some edges to be incident to only one node.  We call such single-node edges 
\textit{boundary rays} and call all other edges \textit{internal edges}.  A planar \textit{embedding} of $G$ 
is a map $\phi$ from the set $V$ of nodes of $G$ into a plane, together with the induced map from internal edges $\{u,v\}\in E$ to the line segments 
$[\phi(u), \phi(v)]$ on the plane, and the map from boundary rays $\{v\}\in E$ to half-infinite rays emanating from points $\phi(v)$.   
Note that we allow different nodes of $G$ to map to the same point on the 
plane; and we allow some edges of $G$ to map to line segments of length zero.

For a positive integer $n \in \N$, the \textit{honeycomb graph} $H_n$ is a certain 3-valent bipartite graph with $3n$ boundary rays, 
see Figure \ref{fig:H_n} for a definition-by-picture. 

\begin{figure}
    \begin{center}
    \includegraphics[scale=0.35]{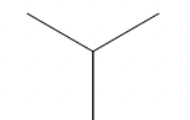}
        \includegraphics[scale=0.35]{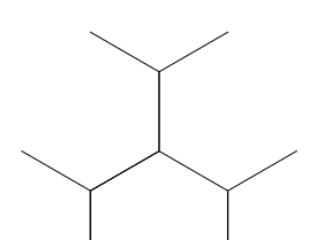}
        \includegraphics[scale=0.35]{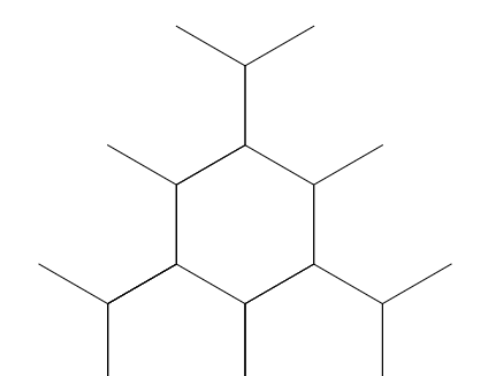}
        \includegraphics[scale=0.35]{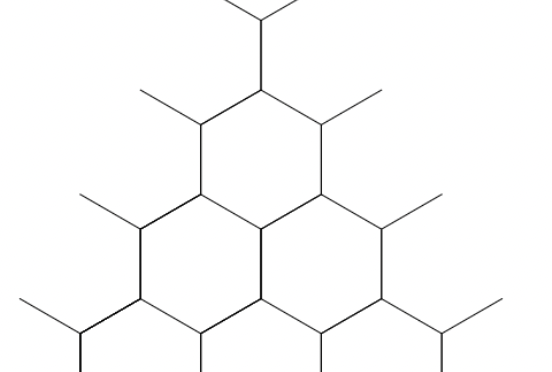} \includegraphics[scale=0.35]{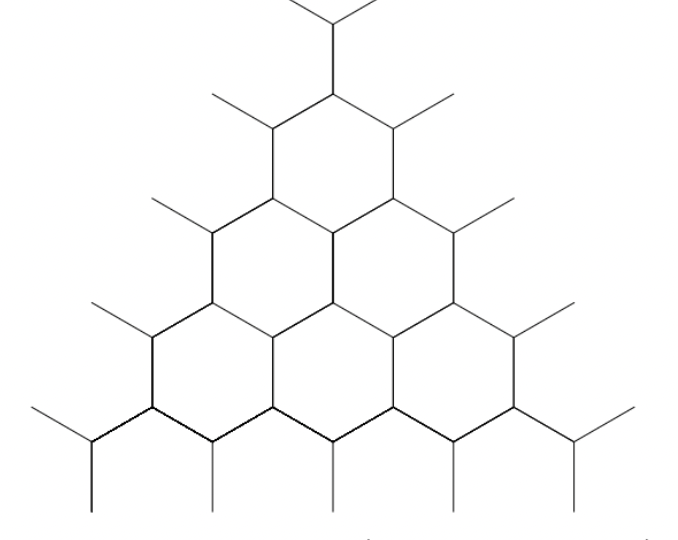}

    \end{center}
    \caption{From left to right are the graphs $H_1$, $H_2, H_3, H_4$, and $H_5$.}
    \label{fig:H_n}
\end{figure}

\begin{definition}[First definition of honeycombs, \cite{honeycombs1}]\label{def:honeycomb1}
    A \textit{honeycomb} is an embedding of $H_n$ into the the plane $\R^3_{\sum = 0}$ satisfying three properties: 
    \begin{enumerate}
        \item Each edge maps to a line segment (possibly of length zero) or a ray in one of the $6$ compass directions
        \item For each node $v$, the three edges adjacent to $v$ map to segments in directions $S, NW, NE$, or in directions $N, SW, SE$ 
        (reading edges as directed away from $v$)
        \item The $3n$ boundary rays map to rays in directions $S, NE$ or $NW$ (no North, Southwest, or Southeast boundary rays are allowed)
    \end{enumerate}
\end{definition}

Sometimes, we refer to a honeycomb which is an embedding of $H_n$ an $n$-honeycomb.

\begin{figure}[h]
\begin{center}
    \includegraphics[scale=0.7]{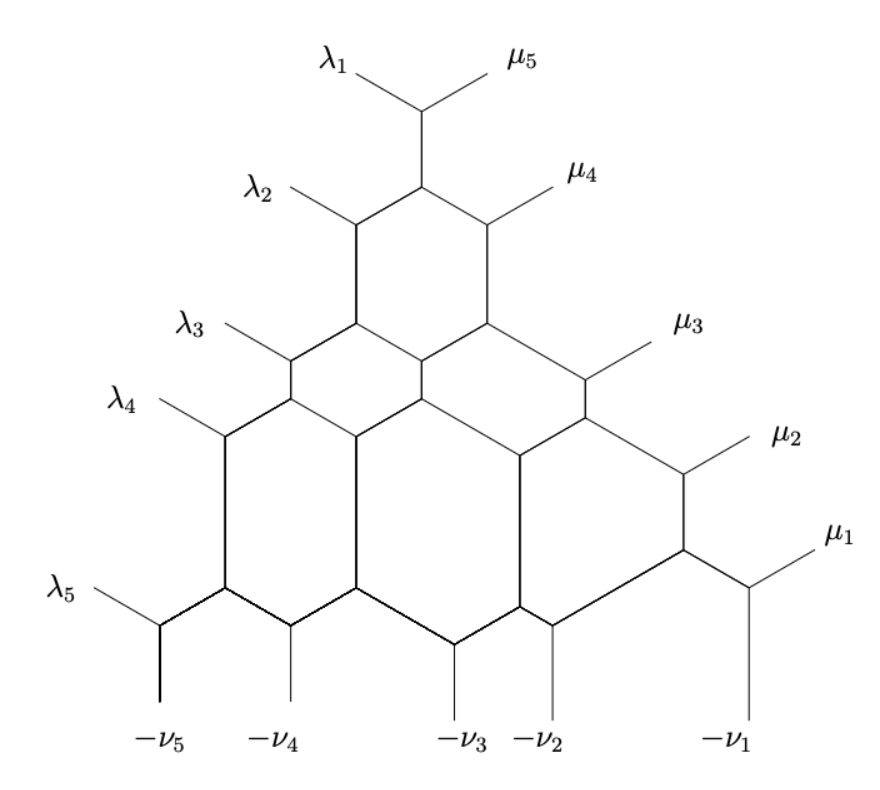}
\end{center}
\caption{An embedding of $H_5$ in which all edge lengths $l_e$ are nonzero, with associated partitions $(\lambda, \mu, \nu)$. Boundary rays are labeled by their positions.}
\label{fig:h5_degen}
\end{figure}

\begin{remark}
    In Definition \ref{def:honeycomb1}, some edges $e$ of $H_n$ are allowed to map to degenerate edges with length $l_e=0$. These `degenerate' honeycombs are combinatorially the most interesting cases. 
\end{remark}

All line segments and rays of a honeycomb lie on lines of the 3 types: 
$(a,*,*) := \{(x,y,z)\in\R^3_{\sum=0} \mid x = a\}$,
$(*,b,*) := \{(x,y,z)\in\R^3_{\sum=0} \mid y = b\}$,
and $(*,*,c) := \{(x,y,z)\in\R^3_{\sum=0} \mid z = c\}$.    We call real numbers $a$, $b$, and $c$ 
(i.e., the constant coordinates) the \textit{positions} of the corresponding 
line segments/rays.

How are honeycombs associated with partitions?  To each $n$-honeycomb we assign a triple $(\lambda,\mu,\nu)$ of partitions
$\lambda,\mu,\nu\in\mathfrak{C}_n$, 
where $\lambda_1,\dots,\lambda_n$ are the positions of its Northwest boundary rays,
$\mu_1,\dots,\mu_n$ are the positions of its Northeast boundary rays, and
$-\nu_1,\dots,-\nu_n$ are the positions of its South boundary rays.

Now that we have defined honeycombs, we fix a notational shortcut that parallels our notation $\Horn_n(\lambda,\mu)$.

\begin{definition}
    For $\lambda, \mu\in\mathfrak{C}_n$, $\Honey_n(\lambda, \mu)$ is the set of all $\nu \in \mathfrak{C}_n$ such that there exists an $n$-honeycomb, whose Northwest, Northeast, and South boundary rays have positions given by the entries of $\lambda, \mu, -\nu$, respectively. 
\end{definition}

We can now reformuate Theorem~\ref{thm:main_thm}, as follows.

\begin{theorem}\label{thm:main_reformulation}
For any $n\geq 1$ and any $\lambda,\mu\in\mathfrak{C}_n$, we have 
$$
\Horn_n(\lambda,\mu)=\Honey_n(\lambda,\mu).
$$
\end{theorem}

There are two (linearly equivalent) ways to parametrize honeycombs.   
We can parametrize an $n$-honeycomb by the array of positions $(p_e)$ assigned to all edges $e$ of $H_n$ (including the boundary rays).
Alternatively, for given $\lambda,\mu,\nu$, a honeycomb is uniquely determined by the array $(l_e)$ of edge lengths $l_e$ 
(divided by ${1\over\sqrt{2}}$) for all internal edges $e$ of $H_n$.
The arrays $(l_e)$ of numbers parametrizing honeycombs are exactly the Berenstein-Zelevinsky patterns.

\begin{definition}[\cite{BZ}]\label{def:BZ_patterns}
For a triple $\lambda,\mu,\nu\in\mathfrak{C}_n$, 
a (real-valued) \textit{Berenstein-Zelevinsky pattern} is an array $(l_e)$ of non-negative real numbers $l_e\geq 0$
assigned to internal edges $e$ of the graph $H_n$ such that, for any hexagon in $H_n$, the six numbers $l_1,\dots,l_6$
corresponding to its edges (listed in, say, a clockwise order) satisfy the \textit{hexagon condition:}
\begin{equation}
\label{eq:hex_condition}
l_1 + l_2 = l_4+l_5,\
l_2 + l_3 = l_5 + l_6,\
l_3 + l_4 = l_6 + l_1,
\end{equation}
and, for a pair of adjacent boundary rays in the Northwest direction corresponding to $\lambda_i$ and $\lambda_{i+1}$
and the two edges $e$ and $e'$ in the shortest path between these rays, we have the \textit{boundary condition:}
$$
l_{e} + l_{e'} = \lambda_i - \lambda_{i+1},
$$
and also similar boundary conditions for Northeast and South boundary rays corresponding to parts of $\mu$ and $-\nu$.
\end{definition}

Clearly, honeycombs and Berenstein-Zelevinsky patterns are in one-to-one correspondence.
Notice that Berenstein-Zelevinsky patterns $(l_e)$ are described by linear equations 
(hexagon and boundary conditions) and the inequalities $l_e\geq 0$.
We will use this observation in Proposition~\ref{prop:polytope}, which says that 
the set $\Honey_n(\lambda,\mu)$ of (minus) positions of South rays of honeycombs 
with fixed positions of Northwest and Northeast boundary rays forms a convex polytope. 
(This fact is less immediate when working with the second definition of honeycombs.)

%Along each compass direction, there is always a constant coordinate. Reading off the constant coordinates of the Northwest, Northeast, and then South boundary rays, we obtain a triple of partitions to a honeycomb. 

\subsection{Honeycombs as zero-tension diagrams}\label{def:honeycomb2}

The second perspective on honeycombs also follows \cite{honeycombs1} and the 
notion of web functions from \cite{webs}.
A \textit{diagram} is a collection of line segments, either finite or half-finite, in $\R^3_{\sum=0}$, such that each segment is drawn along one of the compass directions. Each segment is labeled by a positive integer, its \textit{multiplicity}. The half-infinite segments are again called \textit{boundary rays}. 

Fix a diagram $D$. For any point $p \in \R^3_{\sum = 0}$, we say that $p$ is a \textit{zero-tension point of $D$} if, around some sufficiently small neighborhood of $p$, $D$ is a (possibly empty) finite union of segments emanating from $p$, such that the direction vectors of these rays multiplied by their multiplicities sum to $0$.
Equivalently, the segment multiplicities $m_1,\dots,m_6\geq 0$ of a zero-tension point $p$ (for the 6 compass directions listed, say, in the clockwise order) should satisfy Berenstein-Zelevinsky's hexagon condition \eqref{eq:hex_condition}. 
%\begin{equation}
%\label{eq:hex_condition}
%m_1 + m_2 = m_4+m_5,\
%m_2 + m_3 = m_5 + m_6,\
%m_3 + m_4 = m_6 + m_1
%\end{equation}
In Figure \ref{fig:nbds}, we show all the possible small neighborhoods around zero-tension points up to rotations. 

\begin{figure}
    \begin{center}
        \includegraphics[scale=0.7]{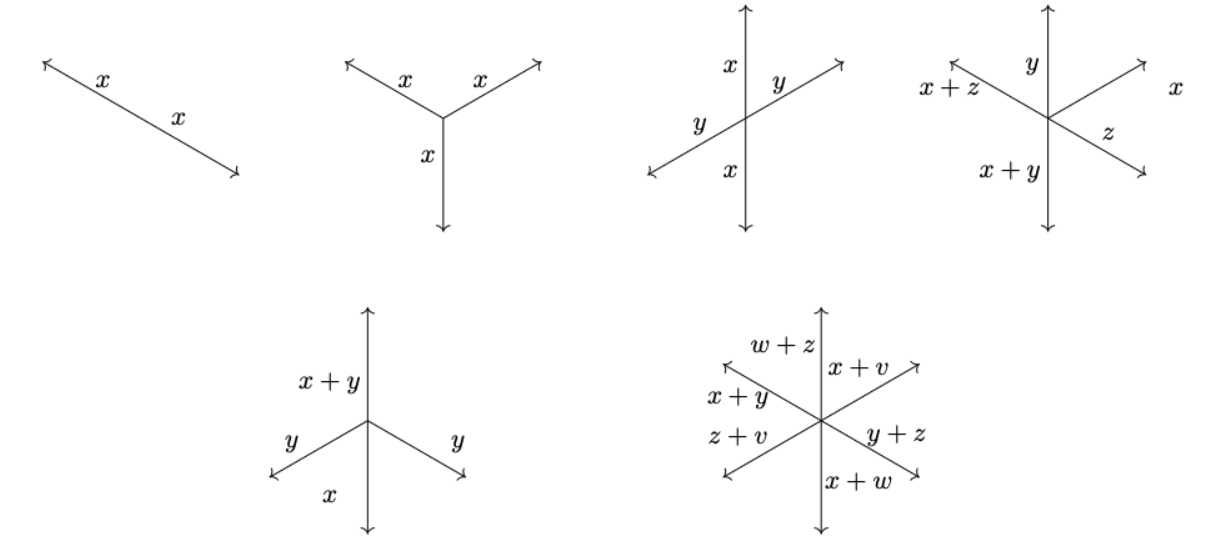}
    \end{center}
    \caption{Image from \cite{honeycombs1}.   Up to rotations, all the possible small neighborhoods around points of a honeycomb. Segments are labeled by their multiplicities.   Actually, all of these types are special case of the last one; they are obtained by setting some of the parameters $x,y,z,v,w$ to zero.}
    \label{fig:nbds}
\end{figure}

\begin{definition}[Second definition of honeycombs, \cite{honeycombs1, webs}]\label{def:honeycomb2}
    A diagram $D$ is a \textit{honeycomb} such that
    \begin{enumerate}
        \item Every point in $\R^3_{\sum = 0}$ is a zero-tension point of $D$
        \item There are only finitely many vertices of $D$ (points with more than two rays emanating outwards) and 
        \item All boundary rays are in the South, Northwest, and Northeast directions
    \end{enumerate}
\end{definition}

For every diagram, there is a corresponding measure on $\R^3_{\sum = 0}$, given by summing the Lebesgue measures on each segment multiplied by the corresponding multiplicities.
We consider two diagrams \textit{equivalent} if their measures coincide outside of a finite collection of points.
The notion of this measure gives us a natural way of \textit{adding} two diagrams: namely, sum their associated measures. This addition of diagrams will be key for us in Section \ref{sec:honeycomb_axioms}.

%It would be slightly more accurate to call a honeycomb an \textit{equivalence class} of diagrams, where two diagrams are considered equivalent if they induce the same measures on $\R^3_{\sum = 0}$. 
%This distinction will not be important for us. 

%An advantage of Definition \ref{def:honeycomb2} is that one can completely characterize small neighborhoods of honeycombs.  The following 

%two easy claims appear in \cite{honeycombs2}, cf.~\cite{webs}.

%\begin{proposition}
%\label{prop:local_char}
%    Let $H$ be a honeycomb according to Definition \ref{def:honeycomb2}. For every point $p \in \R^3_{\sum = 0}$, there is a neighborhood of $p$ in which $D$ is one of the configurations shown in Figure \ref{fig:nbds}.
%\end{proposition}

\begin{proposition}
\label{prop:equiv_defs}
    Definitions \ref{def:honeycomb1} and \ref{def:honeycomb2} are cryptomorphic.  
    More precisely, for any embedding of $H_n$ 
    as in Definition \ref{def:honeycomb1} there exists a unique zero-tension diagram as in Definition \ref{def:honeycomb2}, and vice versa.
\end{proposition}

\begin{proof}
The fact that the measure of any embedding of $H_n$ is the measure of a zero-tension diagram is immediate from the definitions.
The more interesting direction of Proposition \ref{prop:equiv_defs} is finding an embedding of $H_n$ for every zero-tension diagram. 
For a 6-tuple of nonnegative integers $m_1,\dots,m_6\geq 0$ satisfying the hexagon (or tero-tension) condition~\eqref{eq:hex_condition}, let $H_{m_1,\dots,m_6}$
be the `honeycomb-like' graph that has $m_1,\dots,m_6$ boundary rays in all six compass directions.    (The graph $H_{m_1,\dots,m_6}$ is 
similar, but more general than the honeycomb graph $H_n = H_{n,0,n,0,n,0}$.)
In order to convert a zero-tension diagram into an embedding of $H_n$, we need to replace every point $p$ with outgoing segment multiplicities
$m_1,\dots,m_6$ by the embedding of $H_{m_1,\dots,m_6}$ with all edges collapsed to a single point.
We give an example of this process in Figure \ref{fig:def1todef2}. 
\end{proof}

\begin{figure}
    \begin{center}
        \includegraphics[scale=0.75]{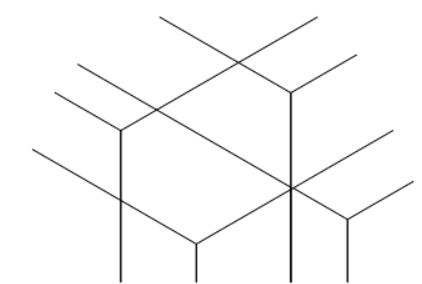}
        \includegraphics[scale=0.75]{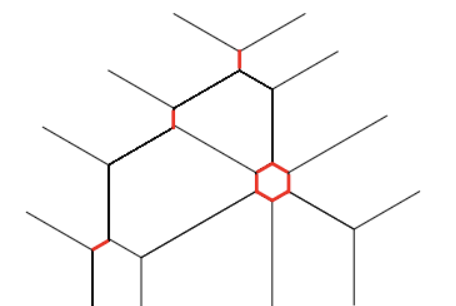}

    \end{center}
    \caption{On the left is a honeycomb satisfying Definition \ref{def:honeycomb2} (a zero-tension diagram). A priori, it does not look like an embedding of $H_4$ because some intersections of segments are not trivalent. At each such intersection, imagine adding short segments, `resolving' each non-trivalent intersection. We get the honeycomb on the right, which is more clearly an embedding of $H_4$. By degenerating the shortest segments in the right honeycomb to $0$, we can recover the honeycomb on the left.}
    \label{fig:def1todef2}
\end{figure}

\subsection{Frameworks}

Frameworks are closely related to honeycombs, though they are not exactly the same. 

\begin{definition}\label{def:frameworks}  
    A \textit{framework} is an embedding  of a finite simple 3-valent biparite graph (with nodes colored red and blue) into the plane $\R^3_{\sum = 0}$, such that
    \begin{enumerate}
        \item Edges are drawn along the compass directions.
        \item All edges have \textit{strictly positive} lengths.
        \item All red nodes are adjacent to a Northwest, Northeast, and South edge;
        and all blue nodes are adjacent a Southwest, Southeast and North edge (with directions read as pointing away from the node).
        \item Finally, all boundary rays are in the Northwest, Northeast, and South directions.
    \end{enumerate}
\end{definition}

\begin{figure}[h]
\centering
    \includegraphics[scale=0.7]
    {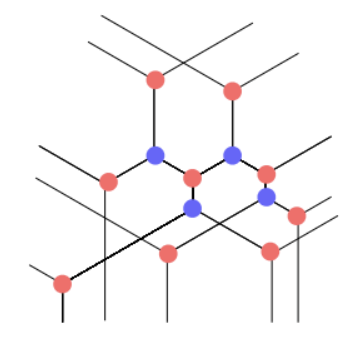}
    \caption{A framework with colored nodes; we give the nodes two colors to emphasize the fact that frameworks are bipartite. Notice that there are $6$ line segment intersections which are not nodes of the framework. As a graph, this particular framework is a tree.} \label{fig:framework_ex}
\end{figure}

Let us point out some subtle but important differences between Definition~\ref{def:frameworks} and prior definitions of honeycombs.
\begin{itemize}
    \item In Definition \ref{def:honeycomb1}, edge lengths $l_e$ were allowed to degenerate to zero. In Definition \ref{def:frameworks}, all edge lengths are strictly positive. 
    \item In Definition \ref{def:frameworks}, nodes are not the same thing as geometric intersections of segments. Every node lies on an intersection of segments, but not every intersection of segments yields a node.    The underlying graph of a framework is not necessarily a planar graph.
    \item In Definition \ref{def:honeycomb2} of honeycombs, each segment came with a multiplicity. Frameworks are given by simple graphs
    \textit{without} multiple edges. However, distinct edges are allowed to geometrically coincide or overlap, as are distinct nodes. 
\end{itemize}

These differences may seem minor, but will actually be significant in Section \ref{sec:honeycomb_axioms}. In general, frameworks are useful objects for studying degenerate honeycombs. 

A framework induces a measure on $\R^3_{\sum=0}$ in the same way as a diagram, and we say that a framework \textit{represents} a honeycomb $H$ if it induces the same measure as $H$. 

\begin{proposition}\label{prop:frameworks'}
    Every framework represents some honeycomb. Every honeycomb is represented by at least one framework. 
\end{proposition}

\begin{proof}
    This proof is simplest when we view honeycombs as zero-tension diagrams (Definition \ref{def:honeycomb2}). One direction is easy: to go from a framework to a honeycomb, replace overlaps between $m$ edges with a single segment with multiplicity $m$, and forget the data of all nodes. We just need to verify that the resulting diagram is zero-tension. All nodes of frameworks yield zero-tension points (because nodes are always trivalent), all points along edges are zero-tension, and overlapping any combination of nodes and edges thus results in a zero-tension point. 

    To go from a honeycomb to a framework, we use the local characterization of zero-tension points shown in Figure~\ref{fig:nbds}.
     In a neighborhood of each point $p$, we need to decide which segments `pass through' $p$ as edges, and which segments are `diverted' through a node. Figure~\ref{fig:neighborhood_resolutions} shows how to do this in the most general case of a degree-six segment intersection.   All other cases are obtained from this case by setting some of the parameters $x,y,z,v,w$ to zero. 
    \begin{figure}[H]
    \begin{center}
 \includegraphics[scale=0.6]{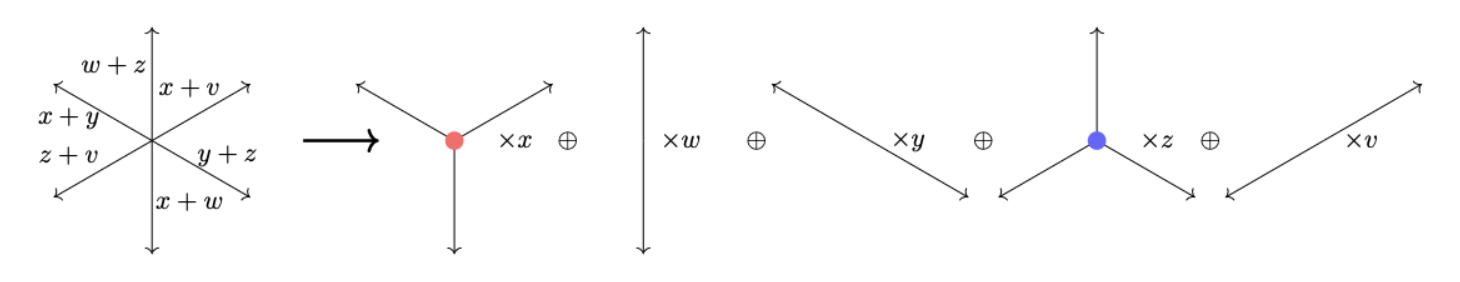}

    \end{center}
    \caption{Breaking up a neighborhood of a zero-tension diagram into edges and trivalent nodes of a framework.}
    \label{fig:neighborhood_resolutions}
\end{figure}

`Resolving' every neighborhood in this way, and then picking an arbitrary way to connect multiple edges 
    gives us a framework. 
\end{proof}

\begin{remark}
In general, the correspondence between honeycombs and frameworks is \textit{not} one-to-one.   When we convert a honeycomb into a framework,
%(as we discussed in the above proof), 
there can be several ways to split a neighborhood of a point in a zero-tension diagram into a collection of 3-valent nodes and edges; moreover, for each 
edge of multiplicity $m$ in a zero-tension diagram there are $m!$ ways to connect the corresponding edges of a framework.   
%These choices give 
%an upper bound for the number of frameworks corresponding to a honeycomb.

For example, the honeycomb shown on Figure~\ref{fig:h5_degen} corresponds to only one framework, because it is an embedding of graph $H_5$ with all nonzero edge
lengths $l_e>0$.  However, the honeycomb shown on the left of Figure~\ref{fig:def1todef2} corresponds to two different frameworks.
Using the notation from Figure~\ref{fig:neighborhood_resolutions}, the 6-valent vertex in this honeycomb corresponds
to $(x,y,z,v,w)$ equal to $(1,0,1,0,0)$ and also to $(0,1,0,1,1)$.  Thus there are 2 ways to `resolve' the 6-valent vertex: (a) as 
a pair of 3-valent nodes, or (b) as 3 crossing edges.
\end{remark}

\begin{definition}\label{def:frame}
For $\lambda,\mu\in\mathfrak{C}_n$, $\mathrm{Frame}_n(\lambda, \mu)$ is the set of all
$\nu\in\mathfrak{C}_n$ such that there exists a framework with $3n$ boundary rays,
whose Northwest, Northeast, and South boundary rays have positions given by
the entries of $\lambda,\mu,-\nu$,
respectively.
\end{definition}

The following claim is an immediate consequence of Proposition~\ref{prop:frameworks'}.

\begin{corollary}\label{cor:honey=frame}
    For any $n\geq 1$ and any $\lambda,\mu\in\mathfrak{C}_n$, we have
    $$
    \Honey_n(\lambda,\mu)=\mathrm{Frame}_n(\lambda,\mu).
    $$
\end{corollary}

Thus, in principle, the notion of a framework can replace the notion of a honeycomb.   

\begin{remark}
%From a certain point of view, frameworks might be more fundamental objects than honeycombs.  
Honeycombs are inherently planar objects.
But, a priori,  there is nothing in the Horn problem that suggests that it should be related to planarity.
On the other hand, frameworks are given by simple 3-valent bipartite graphs $G$ without planarity requirement.
%A priori, $G$ is an arbitrary simple 3-valent bipartite graph.    
%For each $e$ edge of $G$, there is parameter $p_e$, 
%its position.  These parameters $p_e$ should satisfy the conservation law:  the sum of positions $p_e$'s at each node is zero. 

In a sense, frameworks are reminiscent of \textit{Feynman diagrams} that represent \textit{scattering processes} for systems of interacting elementary particles.  From a physics perspective, positions $\lambda_i$, $\mu_i$, $-\nu_i$ of boundary rays are analogues of  
\textit{measurements} of momenta of particles,
%positions $p_e$ of internal edges represent momenta of virtual particles
nodes of two types correspond elementary particle interactions, and the trace 
condition $\sum \lambda_i + \sum \mu_i = \sum \nu_i$ correspond to the  \textit{momentum conservation law.}  
Here we do not want pretend that frameworks correspond to an actual theory in physics.  Note, however, that a similar
``pseudophysics'' language of interacting particles was used in the discussion of web diagrams in \cite{webs}.   
Later, web diagrams morphed into \textit{plabic graphs} \cite{tpgrass}, 
which turned out to be exactly the same combinatorial objects as \textit{on-shell diagrams} that compute  
the \textit{scattering amplitudes} in $\mathcal{N}=4$ supersymmetric Yang-Mills theory, see \cite{GrassAmpl}.   
It would be exciting to see if frameworks can be linked to an actual model in physics.
%We would not be very suprised if frameworks would also be related
%to a certain physics model.
%certain type of non-planar on-shell diagrams.
\end{remark}

\subsection{A concise proof of the Saturation Conjecture}

In this section, we use frameworks to prove the Saturation Conjecture.  As discussed earlier, the Saturation Conjecture was first proved by 
Knutson and Tao \cite{honeycombs1} using honeycombs; and it was the last step of the solution of the Horn Conjecture.
Note that our approach to proving Theorem~\ref{thm:main_thm} does not rely on the Saturation Conjecture.  
The rest of the paper is independent of this section.   We give this proof in
order to illustrate the difference between honeycombs and frameworks.
The main idea of the proof is the same as the idea of the original proof by Knutson and Tao \cite{honeycombs1}.   However, the notion of frameworks allows us to 
shorten and streamline the argument.

The Saturation Conjecture says that, for a triple of \textrm{integer} partitions $\lambda,\mu,\nu\in\mathfrak{C}_n\cap \Z^n$, if there exists a positive integer $k$ such 
that the Littlewood-Richardson coefficient $c_{k\lambda, k\mu}^{k\nu}$ is non-zero, then the Littlewood-Richardson coefficients $c_{\lambda\mu}^\nu$ is also 
non-zero.    According to Berenstein and Zelevinsky \cite{BZ}, the Littlewood-Richardson coefficent $c_{\lambda\mu}^\nu$ equals the number of integer
Berenstein-Zelevinsky patterns with boundary conditions given by $\lambda$, $\mu$, $\nu$.  Equivalently, $c_{\lambda\mu}^\nu$ is the number
of \textit{integer} honeycombs with positions of boundary rays given by parts of $\lambda$, $\mu$, $-\nu$, as discussed above.
Here a honeycomb is \textit{integer} if positions of all its edges are integer.
(Recall that \textit{position} of an edge is the constant coordinate of the corresponding line segments in $\R^3_{\sum=0}$.)
The Saturation Conjecture follows from the following result.

\begin{theorem}[Knutson-Tao \cite{honeycombs1}]\label{thm:saturation} If there exists a honeycomb with integer positions of boundary rays, 
then there exists an integer honeycomb with the same positions of boundary rays.
\end{theorem}

\begin{proof}
Assume that $\nu\in\Honey_n(\lambda,\mu)$ and all parts of $\lambda$, $\mu$, $\nu$ are integer.    Among many different frameworks
with positions of boundary rays given by $\lambda$, $\mu$, $-\nu$, let us pick a framework $F$ with the minimal possible number of nodes.
We claim that the underlying graph $G$ of the framework $F$ is a forest.

Let us prove this claim by contradiction.  Assume that the graph $G$ has a cycle $C$.   
%(The cycle $C$ should have an even length because $G$ is a biparite graph.)
Since a framework has \textit{strictly} positive edge lengths, for a sufficiently small real number $\epsilon>0$, we can shorten or lengthen some edges of $C$ and 
also the other edges incident to vertices of $C$ by $\epsilon$, so that the resulting deformed embedding of $G$ is a valid framework.  Informally, we can call
this deformation of framework `breathing' of the cycle $C$.   Let us make this `breathing' more precise using positions of edges in $C$. 
Pick a way to orient the cycle $C$.   If an edge $e\in C$ points in South, Northwest, or Northeast direction, then increase its position by $\epsilon$;
if an edge $e'\in C$ points in North, Southeast, or Southwest direction then decrease its position by $\epsilon$; positions of all other edges are preserved.
Let us now keep increasing $\epsilon$ until at least one of the edges of the framework collapses to a single point.  
(The collapse happens when $\epsilon \to$ the minimal edge length among all edges of the framework which shorten during this breathing.) 
This collapsed embedding gives a framework with the same positions of all boundary
rays and a smaller number of nodes.   (The number of nodes should decrease by at least $2$.)   Thus we get a contradiction.

Now that we know that the graph $G$ is a forest, we can express (by induction by removing the leaves of $G$) the positions of all edges of the
framework $F$ as integer linear 
combinations of positions of the boundary rays.   It follows that all positions of all edges should be integer.   Thus the framework $F$ corresponds to 
an integer honeycomb.
\end{proof}

\begin{remark}
    The original proof from \cite{honeycombs1} is based on a similar `breathing of cycle' argument.  We will also use a similar `breathing of path' procedure to prove 
    Theorem~\ref{thm:prop:breathing_paths} in Section~\ref{sec:honeycomb_axioms}.
%\end{remark}
%\begin{remark}  
%Note that frameworks are more convenient to use than honeycombs to describe the `breathing' procedure above and to 
%prove Theorem~\ref{thm:prop:breathing_paths} below.  
Note that the space of frameworks with a fixed graph $G$ is a certain polyhedral cone; and   
`breathing' gives just a straight line path in this cone.
%; the graph $G$ remains the same and edge collapse happens only at the last moment.
%when some edge/s of $G$ collapse(s).  
The space of honeycombs also forms a polyhedral cone, cf.~Proposition~\ref{prop:polytope},
or a union of cones if we think of honeycombs as zero-tension diagrams.
However, `breathing' may not give a straight line path in these spaces, but 
rather a piecewise-linear path.   
%Zero-tension diagrams (Definition~\ref{def:honeycomb2}) with a fixed combinatorial structure also form
%a polyhedral cone; so the space of all zero-tension diagrams can be viewed as a certain union of polyhedral cones.  Again, `breathing' gives a piecewise-linear,
%but in general not linear, path in the space of zero-tension diagrams.   
During `breathing', some edges of honeycombs may collapse or uncollapse and some multiple segments can be created or destroyed; see Figure~\ref{fig:breathing_of_path}.
But these collapses and multiple segments are \textit{not} a part of the structure of framework.  
In a sense, frameworks are more `flexible' objects than honeycombs.
\end{remark}

\begin{figure}\label{fig:cycle_breathing}
\centering
    \includegraphics[scale=0.85]{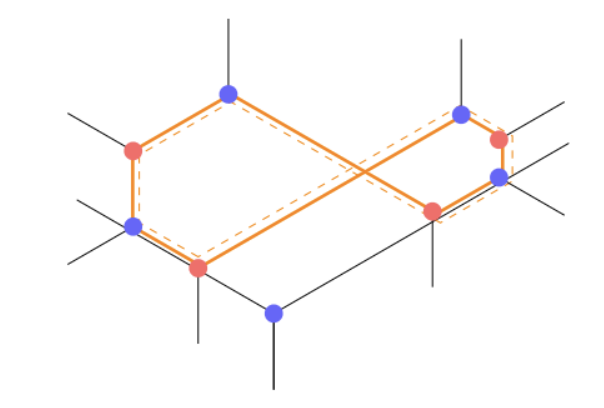}
    \caption{An example of a cycle within a framework which can `breathe,' transforming into the dotted line (only part of the framework is depicted). Notice that, in this example, breathing the cycle breaks segment multiplicities in the corresponding zero-tension diagram.}
\end{figure}

\begin{comment}
\begin{remark}
The fact that, for given $\lambda,\mu,\nu$, frameworks with a minimal number of nodes are given by graphs without cycles,
see proof of Theorem~\ref{thm:saturation} above, also has physics connotations.   In physics, the leading terms
in scattering amplitudes, or \textit{tree-level amplitudes}, are also given by acyclic graphs.
\end{remark}
\end{comment}

\subsection{An example: permutation honeycombs}

There is a special class of honeycombs worth pointing out. A \textit{permutation honeycomb} consists of $n$ $1$-honeycombs overlaid on top of each other. Pictorially, a permutation honeycomb looks like a collection of overlapping $Y$s; see Figure \ref{fig:permutation}. 

Given partitions $\lambda, \mu$, such honeycombs are defined by permutations $\sigma \in S_n$ as follows. Suppose that $\sigma(i) = j$; then, draw a $Y$-shape whose Northwest ray has constant coordinate $\lambda_i$ and whose Northeast ray has constant coordinate $\mu_j$. In Figure \ref{fig:permutation}, $\lambda_1$ is matched to $\mu_3$ by a $Y$, $\lambda_2$ is matched to $\mu_1$, and $\lambda_3$ is matched to $\mu_2$, so this honeycomb is associated to the permutation 
$(\sigma(1),\sigma(2),\sigma(3)) = (3, 1, 2)$.

\begin{figure}
    \begin{center}
        \includegraphics[scale=0.85]{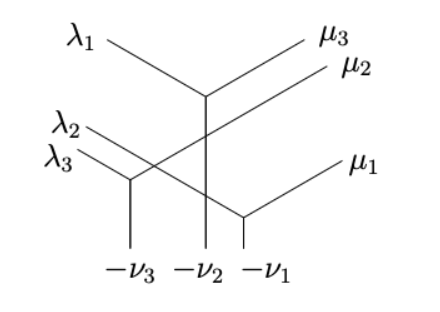}
    \end{center}
    \caption{A permutation honeycomb.
    %associated with the permutation $\sigma(1) = 3, \sigma(2) = 1, \sigma(3) = 2$.  
    Here $\nu_1=\lambda_2+\mu_1$, $\nu_2 = \lambda_1 + \mu_3$, $\nu_3= \lambda_3 + \mu_2$}
    \label{fig:permutation}
\end{figure}

Why are permutation honeycombs important? First, in Section \ref{sec:honeycomb_axioms}, we will see that every vertex of $\Honey_n$ strictly in the interior of $\mathfrak{C}_n$ is given by a permutation honeycomb. (Although not every permutation honeycomb yields a vertex!) Second, permutation honeycombs correspond to the simplest possible Horn triples: simultaneously diagonalizable matrices $A, B, C$ with diagonal entries $\lambda, \mu, \nu$ respectively. By permuting the diagonal entries of $B$, we generate several different spectra of $A+B$. 

\section{Axioms for Horn sequences}\label{sec:axioms}

In our proof of Theorem \ref{thm:main_thm}, we identify four key properties that honeycombs and Horn triples have in common. By a simple inductive argument, any two objects having these four properties in common must be the same. In this section, we state these four properties, which we think of as an axiomatic characterization for Horn triples. In the rest of the paper, we show that both honeycombs and Horn triples satisfy our axioms. 

Recall a few basic notions from convexity theory.   
A point $a\in A$ of a subset $A\subset\R^n$ is called an \textit{extreme point} of $A$ 
if there does not exist a nonzero vector $b\in\R^n$ such that the line segment $[a-b,a+b]$ belongs to $A$.   
According to the Krein-Milman theorem, every compact
convex set $A\subset \R^n$ equals the convex hull of its extreme points.   
A \textit{convex polytope} $P\subset \R^n$ is a compact convex set with finitely many
extreme points, which are called \textit{vertices} of $P$.  Thus a convex polytope $P$ is the convex hull of its vertices.

Recall that $\mathfrak{C}_n$ is the fundamental Weyl chamber $\{ \lambda \in \R^n \mid \lambda_1 \geq \lambda_2 \dots \geq \lambda_n\}$.
%Let $\N=\{1,2,3,\dots,\}$.
%For a subset $I\subset [n]$, $I^c:= [n]\setminus I$ is the complementary subset.

%Both $\Horn_n$ and $\Honey_n$ are partition functions of dimension $n$. 

\begin{definition}\label{def:Horn_sequence}
Let $\{f_n\}_{n\geq 1}$ be a sequence of functions, where $f_n$ is a function from 
$\mathfrak{C}_n\times \mathfrak{C}_n$ to nonempty subsets 
of $\mathfrak{C}_n$.
The sequence $\{f_n\}$ is a \textit{Horn sequence} if it satisfies the following four axioms:
    \begin{enumerate}
        \item (Base case) 
        For $n=1$,  we have $f_1(\lambda,\mu) = \{\lambda + \mu\}$.
        %For any $\nu\in f_n(\lambda,\mu)$, 
        %we have $\sum \lambda_i + \sum \mu_i = \sum \nu_i$.
        %For all $(\lambda, \mu) \in \mathfrak{C}_1 \times \mathfrak{C}_1$, $f_1(\lambda, \mu) = \{\lambda+\mu\}$
        \item (Direct sum) If $\nu \in f_n(\lambda, \mu)$ and $\nu' \in f_m(\lambda', \mu')$, then $\nu \oplus \nu' \in f_{n+m}(\lambda \oplus \lambda', \mu \oplus \mu')$.
        \item (Convexity) For all $n, \lambda, \mu$, $f_n(\lambda, \mu) \subset \mathfrak{C}_n$ is a convex polytope.
        \item (Splitting)  For $n\geq 2$, if $\nu$ is an extreme point of $f_n(\lambda, \mu)$, then 
        there exist
        $\lambda',\mu',\nu'\in\mathfrak{C}_m$ and $\lambda'',\mu'',\nu''\in \mathfrak{C}_{n-m}$, for some $0<m<n$, such that 
         \begin{enumerate}
         \item $\lambda=\lambda'\oplus\lambda''$, $\mu=\mu'\oplus\mu''$, $\nu=\nu'\oplus\nu''$,
         \item $\nu'\in f_{m}(\lambda',\mu')$ and $\nu''\in f_{n-m}(\lambda'',\mu'')$.
      \end{enumerate}    
%  there exist proper subsets $I, J, K \subset [n]$ such that:
%       \begin{enumerate}
%          \item $|I|=|J|=|K| =m$ for some $0 < m < n$
%            \item $\nu_K \in f_m(\lambda_I, \mu_I)$ and 
%            \item $\nu_{K^c} = f_{n-m}(\lambda_{I^c}, \mu_{J^c})$
%        \end{enumerate}
        \end{enumerate}
\end{definition}

Our goal will be to show that  $\{\Horn_n\}$ and $\{\Honey_n\} = \{\Frame_n\}$ are Horn sequences. Then, Theorem \ref{thm:main_thm} 
follows from Proposition \ref{prop:axioms} below. 

\begin{proposition}\label{prop:axioms}
Any two Horn sequences are equal. 
\end{proposition}

\begin{proof}
    We will show, for all $n, \lambda, \mu$, that $f_n(\lambda, \mu) \subseteq g_n(\lambda, \mu)$; then the statement follows by symmetry. We proceed by inducting on $n$. The base case follows from the base case axiom.
    %: $f_1((\lambda_1),(\mu_1))=g_1((\lambda_1),(\mu_1))=\{\lambda_1+\mu_1\}$. 
    
    Now, by the convexity axiom, it suffices to show that every vertex $\nu$ of $f_n(\lambda, \mu)$ 
    is contained in $g_n(\lambda, \mu)$. 
    %Let $\nu$ be a vertex of $f_n(\lambda, \mu)$. 
    By the vertex splitting axiom, we have $\lambda=\lambda'\oplus\lambda''$,
    $\mu=\mu'\oplus\mu''$, $\nu=\nu'\oplus\nu''$, $\nu'\in f_{m}(\lambda',\mu')$, and $\nu''\in f_{n-m}(\lambda'',\mu'')$.
    %we find sets $I, J, K \subsetneq [n]$ such that $\nu_K \in f_m(\lambda_I, \mu_J)$ and $\nu_{K^c} \in f_{n-m}(\lambda_{I^c}, \mu_{J^c})$. 
    Inducting on $n$, we obtain
    $$
    \nu'\in g_m(\lambda',\mu')\quad\textrm{and}\quad \nu''\in g_{m-n}(\lambda'',\mu'').
    $$
%    \[\nu_K \in g_m(\lambda_I, \mu_J) \hspace{2em} \nu_{K^c} \in g_{n-m}(\lambda_{I^c}, \mu_{J^c})\]
    Then, applying the direct sum axiom, we get 
    $$
    \nu=\nu'\oplus \nu''\in g_n(\lambda'\oplus \lambda'', \mu'\oplus \mu'') = g_n(\lambda,\mu),
    $$
    %\[\nu = \nu_{K} \oplus \nu_{K^c}  \in g_n(\lambda_I \oplus \lambda_{I^c}, \mu_J \oplus \mu_{J^c}) = g_n(\lambda, \mu)\] 
as needed.
\end{proof}

\begin{remark}
    It is worth pointing out that our axioms are easier to prove for honeycombs than for Horn triples. In Section \ref{sec:honeycomb_axioms}, our proof of each axiom is purely combinatorial; the main new insight is the concept of frameworks, which leads to a clean proof of the splitting axiom. Meanwhile, the proofs of convexity and splitting for Horn triples are more involved. In Section \ref{sec:Horn_triples}, we modify an argument of Horn to show the splitting axiom for Horn triples. 
    We leave the convexity axiom as a black box, following from symplectic geometry. 
In this way, honeycombs can be seen as a simplifying model for Horn triples. Honeycombs share these same four defining properties with Horn triples, but in more direct and clear ways. 
\end{remark}

\begin{remark} 
    The individual properties listed in Definition \ref{def:Horn_sequence} are not new. On the side of Horn triples,  the base case 
    obvious, and 
    %the conservation axiom is
    %exactly the trace condition $\mathrm{tr}(A+B)=\mathrm{tr}(A)+\mathrm{tr}(B)$.   
    the direct sum axiom follows from taking direct sums of matrices.
    Convexity was the first part of Horn's original conjecture, and follows from deep geometric arguments, which we do not recreate here. 
    Horn~\cite{horn} himself proved a statement very similar (although not equivalent) to the splitting axiom. 
    We modify his argument with a couple of tricks to show the splitting axiom. 
    
    On the side of honeycombs and frameworks, the base case is also clear. Direct sums come from \textit{overlaying} two honeycombs, an operation studied by Knutson and Tao. Convexity is immediate for honeycombs defined as graph embeddings, and was noticed by Berenstein-Zelevinsky \cite{BZ}. Our exact statement of the splitting axiom for honeycombs seems to be new, but similar statements are given in \cite{honeycombs1} and \cite{honeycombs2}. 
\end{remark}

\section{$\Honey_n$ satisfy the axioms}\label{sec:honeycomb_axioms}

In this section, we verify the axioms for honeycombs.  We already noticed that 
$\Honey_n = \Frame_n$ in Corollary~\ref{cor:honey=frame}.   Thus to check individual axioms, we can use either honeycombs or frameworks, whichever is 
more convenient.

\subsection{Base case, direct sum, and convexity axioms}

The base case is easy: when $n=1$, the only possible honeycomb is a single $Y$ shape. 
%The conservation axiom is equally to check either for honeycombs or for frameworks.

\begin{comment}
\begin{proposition}  For any framework with boundary rays given by a triple $(\lambda,\mu,\nu)$, we have $\sum \lambda_i + \sum \mu_i = \sum \nu_i$.
\end{proposition}

\begin{proof}  By definition, for every node $v$ of the framework, the sum of positions $p_e$ of three edges incident to $v$ is zero (local conservation).  
By adding these quantities for all red nodes and subtract them for all blue nodes, we get
$$
\sum_{u\textrm{ is red}} \left(\,\sum_{e\textrm{ is indicent to } u} p_e\right)  - 
\sum_{v\textrm{ is blue}} \left(\, \sum_{f\textrm{ is indicent to } v} p_f\right)  = 0.
$$
Since positions $p_e$ of all internal edges $e$ cancel each other in the left hand side, we deduce that the sum of position of $3n$ boundary rays is zero,
which is exactly the needed conservation property.
\end{proof}
\end{comment}

The convexity and direct sum also turn out to be straightforward.   
Note that the convexity of $\Honey_n$ is clear from the first definition of honeycombs, whereas the direct sum axiom is clear from the second definition of honeycombs
as zero-tension diagrams.

\begin{proposition}\label{prop:polytope}
    For all $\lambda, \mu, n$, $\Honey_n(\lambda, \mu)$ is a convex polytope. 
\end{proposition}

\begin{proof}
    The most convenient way to establish convexity is through Definition~\ref{def:honeycomb1} of honeycombs as embeddings of the graph $H_n$ and corresponding Berenstein-Zelevinsky patterns (Definition~\ref{def:BZ_patterns}).
    For fixed $\lambda$ and $\mu$, let $\mathrm{BZ}_n(\lambda,\mu)$ be the set of 
    Berenstein-Zelevinsky patterns $(l_e)$ with given $\lambda$ and $\mu$ and arbitrary $\nu$.
%    arrays $\{x_e\}$ of parameters $x_e\in\R$
%       If $\phi: H_n \to \R^3_{\sum = 0}$ is such an embedding, then for an edge $e=\{u,v\}$ of $H_n$, let $x_e$ be the length of the line segment $[\phi(u),\phi(v)]$
%divided by $\sqrt{2}$. The collection of parameters 
%    $\{x_e\}_{e \in E(H_n)}$, together with positions of the Northwest and Northeast boundary rays, completely determines the embedding $\phi$. 
%    Let $\mathrm{BZ}_n(\lambda,\mu)$ be the set of arrays $\{x_e\}$ of parameters $x_e\in\R$ that correspond to honeycombs with fixed positions of 
%    Northeast and Northwest boundary rays given by parts of $\lambda$ and $\mu$, respectively.
    (The set $\mathrm{BZ}_n(\lambda,\mu)\subset \R^{3n(n-1)/2}$ should not be confused with $\Honey_n(\lambda,\mu)\subset\R^n$.)
%    The set $\mathrm{BZ}_n(\lambda,\mu)$ is described by the following conditions:
%    \begin{enumerate}
%        \item For all edges $e$, $x_e \geq 0$. 
%        \item (Hexagon condition) For all $6$-cycles of $H_n$, with edges labeled clockwise as $e_1, e_2 \dots ,e_6$, we have  %$x_{e_1}+x_{e_2} = x_{e_4}+x_{e_5}$, $x_{e_2} + x_{e_3} = x_{e_5}+x_{e_6}$, and $x_{e_3}+x_{e_4} = x_{e_6}+x_{e_1}$.
%        \item (Boundary condition) For a pair of adjacent Northwest boundary rays with positions $\lambda_i$ and $\lambda_{i+1}$ and the %two edges 
%        $e$, $e'$ of $H_n$ connecting these boundary rays, we have  $\lambda_i - \lambda_{i+1} = x_e + x_{e'}$.   Similarly, we have %analogous 
%        equalities for the Northeast boundary rays, whose positions are given by parts of $\mu$. 
%    \end{enumerate}
%    Remark, that these conditions are exactly the definition of Berenstein-Zelevinsky patterns \cite{BZ}.
    Notice that Berenstein-Zelevinsky patterns are defined by 
    linear equations in the parameters $l_e$ and the inequalities $l_e\geq 0$.
    %that these conditions are linear in the parameters $l_e$.  
    It is easy to see that these conditions imply that 
    $l_e\leq \max(\lambda_1-\lambda_n, \mu_1-\mu_n)$, for any edge $e$; thus the set $\mathrm{BZ}(\lambda,\mu)$ is bounded.
    This shows that $\mathrm{BZ}_n(\lambda,\mu)$ is a convex polytope, called the \textit{Berenstein-Zelevinsky polytope.}
    The set $\Honey_n(\lambda,\mu)$ is the projection of the Berenstein-Zelevinsky polytope $\mathrm{BZ}_n(\lambda,\mu)$ to the positions $-\nu_i$ of the boundary rays in the South direction.
    Thus $\Honey_n(\lambda,\mu)$ is also a convex polytope, because a linear projection of a convex polytope is a convex polytope.
\end{proof}

\begin{proposition}
    $\{\Honey_n\}$ satisfies the direct sum axiom. 
\end{proposition}

\begin{proof}
    Essentially, we can draw one honeycomb on top of the other. More formally, given two honeycombs $H, H'$, their \textit{overlay} $H \oplus H'$, see \cite{honeycombs1}, is given by summing the measures associated to $H, H'$. Here we are using Definition \ref{def:honeycomb2} of honeycombs
    as zero-tension diagrams. If $H$ has boundary rays $(\lambda, \mu, \nu)$ and $H'$ has boundary rays $(\lambda', \mu', \nu')$, then $H \oplus H'$ has boundary rays $(\lambda \oplus \lambda', \mu \oplus \mu', \nu \oplus \nu')$. Figure \ref{fig:overlay} shows an example of this operation. 
\end{proof}

\begin{figure}
    \begin{center}
    \includegraphics[scale=0.5]{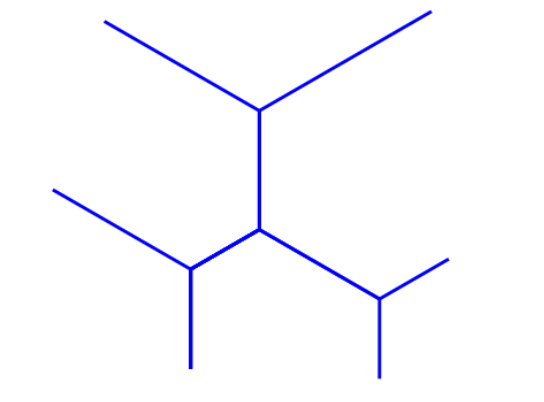}
        \includegraphics[scale=0.5]{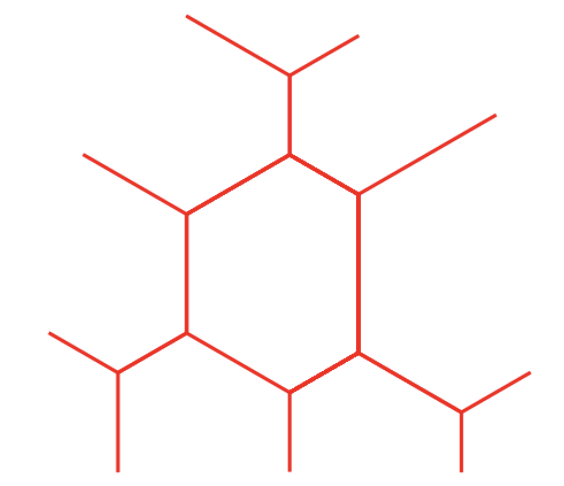}
        \includegraphics[scale=0.5]{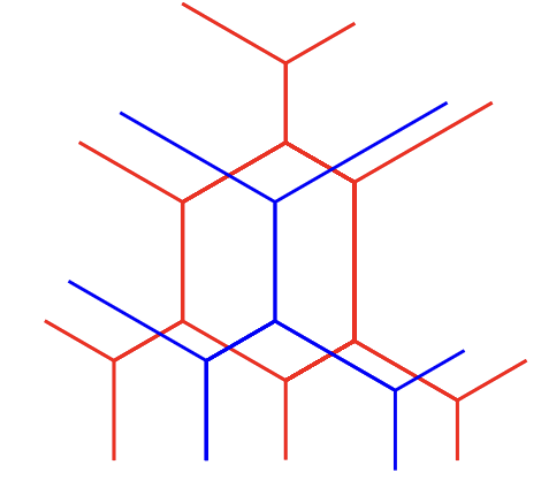}
    \end{center}
    
    \caption{The 5-honeycomb shown on the right is the \textit{overlay} of the blue 2-honeycomb and red 3-honeycomb.}
    \label{fig:overlay}
    
\end{figure}

\subsection{The vertex splitting axiom for $\Honey_n$}

What remains is to prove the vertex splitting axiom for $\Honey_n$, which is a bit more subtle. 

\begin{theorem}\label{thm:prop:breathing_paths}
For $n\geq 2$,
    let $\nu$ be a vertex of the polytope $\Honey_n(\lambda, \mu)$, and let $H$ be any honeycomb with boundary rays $(\lambda, \mu, \nu)$. Then any framework representing $H$ is disconnected as a graph. 
\end{theorem}

This proposition implies the splitting axiom, because the triple $(\lambda,\mu,\nu)$ corresponding to boundary rays of a disconnected framework splits
into direct sum of triples for connected components of the framework. 

Recall that we defined vertices of a polytope $P$ as its extreme points.  By the definition, a point $a\in P$ is \textit{not} a vertex if and only if there exists a nonzero vector $b$ such that $[a-b, a+b]\subset P$.
%Our strategy for proving Proposition \ref{prop:breathing_paths} is to use 
%The following lemma is basically a reformulation of the definition of vertices of a polytope as its extreme points.
%\begin{lemma}\label{lemma:vertices}
%    Let $P \subset \R^n$ be a convex polytope, and let $x \in P$.   There exists a non-zero vector $v \in \R^n$ and $\epsilon>0$ such that $x+\epsilon v \in P$ and $x - %\epsilon v \in P$ if and only if $x$ is not a vertex of $P$. 
%\end{lemma}
To prove Theorem~\ref{thm:prop:breathing_paths}, for a connected framework $F$, 
we want to engineer some way to perturb $F$ that \textit{leaves constant} positions of the Northwest and Northeast boundary rays, but shifts positions of 
some of the South boundary rays back and forth by some small $\epsilon\ne 0$. 
%More specifically, we will choose two parts $\nu_i, \nu_j$ of $\nu$, and find a path $P$ in the framework connecting boundary rays with constant coordinates $\nu_i$ and $\nu_j$. Then we use $P$ to shift $\nu$ by $\pm \epsilon(e_i - e_j)$, where $e_i$ is the $i$th standard basis vector of $\R^n$, by `breathing' $P$ back and forth.  
The strategy is slightly more subtle when $\nu$ has repeated parts: when we perturb some of the South boundary rays, we will want to ensure that all multiplicities of $\nu$ are preserved. Interestingly, perhaps as expected, this same subtlety arises when we study vertex splitting for Horn triples in Section \ref{sec:Horn_triples}. 

\begin{proof}[Proof of Theorem~\ref{thm:prop:breathing_paths}]
The case then all parts of $\nu$ are are equal to each other is easy.  In this case there is (at most) one honeycomb with boundary rays given by 
triple $(\lambda,\mu,\nu)$. Namely, the unique honeycomb in this highly degenerate case consists of several overlapping $Y$-shapes whose South rays have 
coinsiding positions.   (It is actually a permutation honeycomb.)   Its framework is a union of $n$ $Y$-shaped connected components, which is a disconnected graph
for $n\geq 2$.

Now we can assume that $\nu$ has at least two distinct entries $\nu_i\ne \nu_j$.   
Suppose that a honeycomb $H$ with boundary rays given by the triple $(\lambda,\mu,\nu)$ is 
represented by a framework $F$ with connected graph $G$.

First, consider the case when both entries $\nu_i$ and $\nu_j$ of $\nu$ have multiplicity 1.   Since the framework is connected, we can find a path $P$ 
in the graph $G$ that starts at the South boundary ray with position $-\nu_i$ and ends at the South boundary ray with position $-\nu_j$.
    Because each node of $P$ is trivalent and all edges have strictly positive lengths, we can find a small $\epsilon > 0$ such that $P$ can `breathe' back and forth by $\epsilon$, altering $\nu$ by $\pm \epsilon(e_i - e_j)$ but fixing $\lambda$ and $\mu$,
    see Figure~\ref{fig:breathing_of_path}.
    (Here $e_i$ is the $i$th standard basis vector of $\R^n$.) 
Let us make this more precise. Pick a way to direct the path $P$.
If an edge of $P$ points South, Northwest, or Northeast, increase its position (i.e., the constant coordinate of the corresponding line segment 
in $\R^3_{\sum=0}$)  by $\epsilon$. If an edge of $P$ points North, Southeast, or Southwest, decrease its position by $\epsilon$. Suppose that $\epsilon$ is sufficiently small. (It is enough to pick $\epsilon$ smaller than the minimal edge length in the framework
and $\min(\nu_{i-1}-\nu_i, \nu_{i} - \nu_{i+1}, \nu_{j-1}-\nu_j, \nu_{j} - \nu_{j+1}$.)
Then this deformation produces a valid framework with the same positions of Northwest and Northeast boundary rays given by $\lambda$ and $\mu$
and the same ordering of the South boundary rays as in the original framework.   The vector of (minus) positions of the South boundary rays in the deformed framework is
$\nu + \epsilon (e_i - e_j)$.   Similarly, if we shift positions of edges in $P$ in the opposite direction, we obtained a framework with boundary rays 
corresponding the the triple $(\lambda,\mu, \nu - \epsilon (e_i - e_j)$.
This shows that the line segment $[\nu - \epsilon(e_i - e_j), \nu + \epsilon (e_i - e_j)]$ belongs to the polytope $\Honey_n(\lambda, \mu)$.
Thus $\nu$ is not a vertex of $\Honey_n(\lambda, \mu)$.

    \begin{figure}
        \begin{center}
            \includegraphics[scale=1]{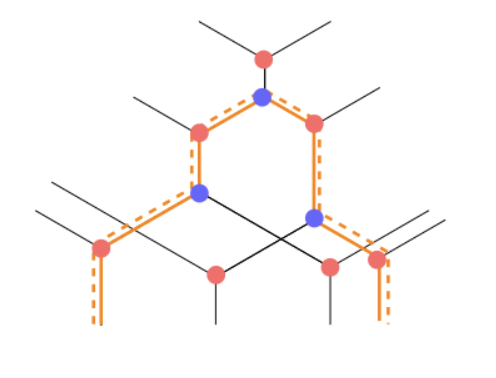}
        \end{center}
        \caption{A framework for $n=4$ with a highlighted path connecting $-\nu_1$ to $-\nu_4$ that can breathe back and forth. }
        \label{fig:breathing_of_path}
    \end{figure}

Let us now extend the agrument to the general case when $\nu$ has two distinct entries $\nu_i\ne \nu_j$ with multiplicities $m$ and $m'$, respectively.
This means that, in the general case, $\nu$ might lie on the boundary of the fundamental Weyl chamber $\mathfrak{C}_n$.  
Now there is a complication in our strategy: namely, as we `breathe path' by $\epsilon$ and deform $\nu_i$ and $\nu_j$ as above,
then at least one of the points $\nu - \epsilon(e_i - e_j)$ or $\nu + \epsilon(e_i - e_j)$ will lie outside of the chamber $\mathfrak{C}_n$ and it will not 
belong to the polytope $\Honey_n(\lambda,\mu)$.   The solution is that we need to deform not just two entries $\nu_i$ and $\nu_j$ but 
\textit{all} entries of $\nu$ that are equal to $\nu_i$ or $\nu_j$ while preserving their multiplicities.

Now the graph $G$ of the framework $F$ has $m$ \textit{different} boundary rays that are embedded into $\R^3_{\sum=0}$ as South rays with the same positions $-\nu_i$;
and $m'$ different rays with positions $-\nu_j$.  Let us call these rays $r_1,\dots,r_{m}$ and $r'_1,\dots,r'_{m'}$, respectively.
Let us pick $m\cdot m'$ paths $P_{a,b}$ in the graph $G$, where $P_{a,b}$ is an arbitrary path starting at $r_a$ and ending at $r_b$,
for $a=1,\dots,m$, $b=1,\dots,m'$.
These paths $P_{a,b}$ are allowed to have edges and nodes in common. (Indeed they must, as there are $m'$ paths through each $r_a$). 

Let us deform the framework $F$ by `breathing' each of the paths $P_{a,b}$ by sufficiently small $\epsilon$ in exactly the same way as we did in the
case of multiplicities 1.   (We can do this breathing of paths $P_{a,b}$ one by one.   The result does not depend on ordering of the paths.)
The resulting framework has the same positions of the Northwest and Northeast boundary rays given by $\lambda$ and $\mu$,
while each South boundary ray with original position $-\nu_i$ is shifted by $m'\epsilon$
and each South boundary ray with original position $-\nu_j$ is shifted by $-m\,\epsilon$.   For sufficiently small $\epsilon$, we get a valid framework with the 
same ordering and the same multiplicities of the South boundary rays.   (To be exact, it is enough to pick $\epsilon < {1\over m\cdot m'} \min(l,d)$,
where $l$ is the minimal edge length in the framework and $d$ is the minimal difference between distinct entries of $\nu.)$  Similarly, we can deform 
the collection of paths $P_{a,b}$ in the opposite direction.
This shows that the polytope $\Honey_n(\lambda,\mu)$ contains the line segment $[\nu - b, \nu + b]$, where
$b$ is the vector $\epsilon\sum_{s\in [i,i+m-1], t\in[j,j+m'-1]} (e_s - e_t)$.
Thus $\nu$ is not a vertex of $\Honey_n(\lambda,\mu)$.
\end{proof}

Using a similar argument, we can also prove the following statement.

\begin{proposition}
    Let $\nu$ be a vertex of the polytope $\Honey_n(\lambda, \mu)$ \textit{strictly} in the interior of $\mathfrak{C}_n$. Then any honeycomb with boundary rays
    corresponding to the triple $(\lambda, \mu, \nu)$ is a \textit{permutation} honeycomb.   
    More generally, if $\nu$ belongs to a $d$-dimensional face of the polytope $\Horn_n(\lambda,\mu)$ and $\nu$ is strictly in the interior of $\mathfrak{C}_n$, then any framework representing the triple $(\lambda,\mu,\nu)$ has at least
    $n-d$ connected components as a graph.
\end{proposition}

This proposition implies that all vertices $\nu$ of $\Honey(\lambda,\mu)$ in the interior of $\mathfrak{C}_n$ have the form
$\nu = u(\lambda) + w(\mu)$, for some permutations $u,w\in S_n$.

\begin{proof}
Suppose that $\nu$ is a vertex of $\Honey_n(\lambda,\mu)$ and all parts of $\nu$ are distinct.   Pick any framework $F$ that represents the 
triple $(\lambda,\mu,\nu)$.   We claim that any connected component of the graph of the framework $F$ must be Y-shaped, i.e., it must correspond to a
$1$-honeycomb.
By contradiction, assume that there is a connected component that is not Y-shaped.   Then it has at least two South boundary rays with positions 
$-\nu_i \ne -\nu_j$.   By breathing a path connecting these rays as the did in proof of Theorem~\ref{thm:prop:breathing_paths}, we deduce
that $\nu$ is not a vertex of $\Honey(\lambda,\mu)$.

The claim about $d$-dimensional faces is proved along the same lines.
\end{proof}
    
\section{$\Horn_n$ satisfies the axioms}\label{sec:Horn_triples}

In this section, we verify the axioms for spectra of sums of Hermitian matrices.

\subsection{Base case, direct sum, and convexity axioms}

%The conservation axiom is exactly the trace condition for sums of Hermitian matrices $\mathrm{tr}(A+B)=\mathrm{tr}(A)+\mathrm{tr}(B)$.
The base case axiom is obvious.
The direct sum axiom is almost as easy. Suppose $(\lambda, \mu, \nu)$ and $(\lambda', \mu', \nu')$ are Horn triples coming from spectra of matrices $A$, $B$, $C=A+B$ and $A'$, $B'$, $C'=A'+B'$, respectively. Then the triple 
$(\lambda \oplus \lambda', \mu\oplus \mu', \nu \oplus \nu')$ is given by the spectra of block-diagonal matrices
$$\begin{bmatrix}
    A & 0 \\ 0 & A'
\end{bmatrix}
, \ 
\begin{bmatrix}
    B & 0 \\ 0 & B'
\end{bmatrix}
, \ 
\begin{bmatrix}
    C & 0 \\ 0 & C'
\end{bmatrix}.
$$

\begin{comment}
For sums of Hermitian matrices the convexity axiom is far from obvious. The fact that $\Horn_n(\lambda, \mu)$ is always a convex polytope was conjectured by Horn
\cite{horn}, and proved using machinery from symplectic geometry: Kirwan's Theorem \cite{kirwan} and the Atiyah-Guillemin-Sternberg Theorem (see \cite{atiyah}), before the full extent of Horn's Conjecture was proved. We direct the interested reader to Knutson's survey \cite{symplectic} for an introduction to the symplectic geometry underlying convexity. In some sense, the fact that we are not able to give a simple proof of convexity seems to be a feature and not a bug, in that it demonstrates how the defining properties of Horn sequences are more immediate and accessible from the honeycomb perspective. 
\end{comment}

%\subsection{Proof outline of convexity and vertex splitting axioms for Horn triples}

%What remains is to prove the vertex splitting axiom. 
For sums of Hermitian matrices, the convexity axiom is far from obvious. The fact that $\Horn_n(\lambda, \mu)$ is always a convex polytope was conjectured by Horn
\cite{horn}, and proved using machinery from symplectic geometry: Kirwan's Theorem \cite{kirwan} and the Atiyah-Guillemin-Sternberg Theorem (see \cite{atiyah}), before the full extent of Horn's Conjecture was proved. We direct the interested reader to Knutson's survey \cite{symplectic} for an introduction to the symplectic geometry underlying convexity.   In some sense, the fact that convexity is more involved for Horn triples than for honeycombs
seems to be a feature and not a bug, 
in that it demonstrates how the defining properties of Horn sequences are more immediate and accessible from the honeycomb perspective. 

What remains is to prove the splitting axiom. 

\subsection{Proof outline of splitting axiom for Horn triples}
\label{ssec:proof_outline_splitting}

We will provide  an `elementary' proof for the splitting axiom.
Note that our proof does not use the convexity axiom.
Our approach is based on an argument of Horn \cite{horn}, with several additions and modifications. Namely, Horn's original proof applied to real symmetric matrices, not Hermitian matrices, and required the parts of $\mu$ and $\nu$ to be distinct. We extend Horn's argument from the real case to the complex case with a simple trick, and further
%use the fact that the set of Horn triples is a closed set 
extend it to partitions with multiplicities.    

%This proves the needed splitting property for $\Horn_n(\lambda,\mu)$.  

%Then we argue that a similar method also implies the convexity of $\Horn_n(\lambda,\mu)$.

Before giving full proofs, we summarize the main ideas. 
Suppose that $\nu$ is an extreme point or, more generally, 
a boundary point of $\Horn_n(\lambda, \mu)$. We will first assume that both $\mu$ and $\nu$ 
have distinct parts, and then deduce the general case.
Similarly to the proof of Theorem~\ref{thm:prop:breathing_paths} for honeycombs, we imagine perturbing $\nu$ back and forth while keeping it within $\Horn_n(\lambda, \mu)$. 
Then we can analyze the tangent space of infinitesimal directions in which $\nu$ can shift.
%We will see that this $V_{\nu}$ is a real vector space (Definition \ref{def:V_{nu}}). 
If $\nu$ is on the boundary of $\Horn_n(\lambda,\mu)$, the tangent space should have dimension \textit{strictly} less than $n-1$.
%(This dimension bound is the only part of our argument that relies on convexity of $\Horn_n(\lambda,\mu)$.)

%(Proposition \ref{prop:dim}). 
%Furthermore, when $\nu$ is assumed to live in lower-dimensional faces of $\Horn_n(\lambda, \mu)$, $V_{\nu}$ should become even lower dimensional. 

How can we perturb $\nu$? Notice that every point of $\Horn_n(\lambda, \mu)$ arises from a choice of unitary matrix $U$. 
In more detail, suppose $A, B, A+B$ is a triple of Hermitian matrices with spectra $\lambda, \mu, \nu$, respectively. 
Let $D_{\lambda}:=\mathrm{diag}(\lambda_1,\dots,\lambda_n)$ be the diagonal matrix with diagonal entries $\lambda_i$. 
We may assume, after an appropriate change of basis, that $B$ is a diagonal matrix and $A$ is diagonalized by a unitary matrix $U$:
$$
A = U D_\lambda U^{-1},\ 
B = D_\mu,\ A+B = U D_\lambda U^{-1} + D_\mu
$$
%the diagonal matrix $A=D_{\lambda}:=\mathrm{diag}(\lambda_1,\dots,\lambda_n)$ with diagonal entries $\lambda$.  Let $U$ be the unitary matrix that %diagonalizes $B$,
%that is, $B=U^{-1} D_\mu U$.  We obtain 
%the following expression for $A+B$:
%\begin{equation}\label{eq:A+B}
%$$
%    A+B = D_{\lambda} + U^{-1}D_{\mu}U
%$$
Thus, by the definition,
\begin{equation}\label{eq:horn=spec}
\Horn_n(\lambda,\mu)=\{\mathrm{spec}(UD_\lambda U^{-1} + D_\mu )\mid U \textrm{ is unitary}\}.
\end{equation}
%This implies that $\Horn_n(\lambda,\mu)$ is a compact connected subset of $\R^n$ that continuously depend on $\lambda$ and $\mu$.
Our goal is to perturb $U$ and study how the spectrum $\nu = \mathrm{spec}(UD_{\lambda}U^{-1} + D_{\mu})$ change. 
The Lie algebra of the unitary Lie group is the algebra of skew-Hermitian matrices $S=-S^*$. Thus,
for a skew-Hermitian matrix $S$, the exponential $e^S$ is unitary. 
We will explicitly compute partial derivatives of the spectrum 
of $(e^{S}U) D_\lambda (e^SU)^{-1} + D_\mu$
at $S=0$. 
Then we will translate the dimension bound for the tangent space to $\Horn_n(\lambda,\mu)$ at $\nu$ into the upper bound 
$\mathrm{rank}(G)\leq n-2$ on rank on a certain matrix $G$ expressed in terms of these partial derivatives.
This bound implies that the associated doubly stochastic matrix $FF^* = I  -  GG^*$ has eigenvalue $1$ of multiplicity $2$ or more. 
%Finally, we will use this fact to deduce the needed splitting property.

Ultimately, our proof of splitting relies on the following classical result that follows from Perron-Frobenius Theorem \cite{perron-frobenius}.
Recall that a (row) \textit{stochastic matrix} is a square matrix with real non-negative entries whose row sums are equal to $1$.  
%Stochastic matrices correspond to Markov chains.  
Clearly, $(1,\dots,1)^T$ is always an eigenvector of a stochastic matrix with eigenvalue 1.

\begin{lemma}\label{lemma:perron-frobenius}
    Let $M$ be a stochastic $n\times n$ matrix with $1$ as an eigenvalue of multiplicity 2 or more. Then $M$ is decomposable. That is, there exists a permutation matrix $P$ and square blocks $M_1, M_2$ such that 
    \[M = P \begin{pmatrix}
        M_1 & 0\\
        0 & M_2
    \end{pmatrix}P^{-1}\]
\end{lemma}

The reason why a result on stochastic matrices can be applied to unitary matrices is based on the following trivial observation:
The matrix whose entries are squares of absolute values of entries of a unitary matrix is a doubly stochastic matrix. 

Let us now give the details.

\subsection{Boundary point splitting}
%Alex: We are NOT assuming that \mu and \nu have distinct parts in the discussion following Proposition 7.2.
%(assuming that $\mu$ and $\nu$ have distinct entries).

Notice that the trace condition $\mathrm{tr}(A+B)=\mathrm{tr}(A)+\mathrm{tr}(B)$
implies that $\Horn_n(\lambda,\mu)$ belongs to the hyperplane 
$H=\{x_1+\cdots+x_n= 0\}\subset\R^n$ translated by vector $\lambda+\mu$.
Let $\lVert x \rVert:=\sqrt{\sum |x_i|^2}$ denote the standard Euclidean norm on $\R^n$.
For $\epsilon>0$, let $\mathrm{B}_{\epsilon}(\nu)$ be the $(n-1)$-dimensional ball 
$$
\mathrm{B}_{\epsilon}(\nu) := \{x\in \R^n\mid \lVert x-\nu\rVert\leq \epsilon, \ x - \nu \in H\},
$$
i.e., $\mathrm{B}_{\epsilon}(\nu)$ is the $\epsilon$-neighborhood of $\nu$ lying in the same translated hyperplane $H$ as $\nu$.

%thus $\Horn_n(\lambda,\mu)$ and all its tangent spaces are at most $(n-1)$-dimensional.

\begin{definition}\label{def:boundary_points}
Define an \textit{interior point} $\nu\in\Horn_n(\lambda,\mu)$ as a point for which there exists $\epsilon>0$ such that the 
intersection of  ball $\mathrm{B}_{\epsilon}(\nu)$
with the fundamental Weyl chamber $\mathfrak{C}_n$ belongs to $\Horn_n(\lambda,\mu)$:
$$
\mathrm{B}_\epsilon(\nu)\cap\mathfrak{C}_n \subset \Horn_n(\lambda,\mu).
$$
Let $\Horn_n(\lambda,\mu)^\circ$ denote the set of all interior points of $\Horn_n(\lambda,\mu)$.

%Define the \textit{boundary}  $\partial\,\Horn_n(\lambda,\mu)\subset \Horn_n(\lambda,\mu)$ of $\Horn_n(\lambda,\mu)$
%is the 

Define a \textit{boundary point} of $\Horn_n(\lambda,\mu)$  
%$\nu\in \partial\,\Horn_n(\lambda,\mu)\subset \Horn_n(\lambda,\mu)$ 
as a point  that not an interior point.  Let 
%Let $\Horn_n(\lambda,\mu)^\circ$ denote the set of all internal points of $\Horn_n(\lambda,\mu)$, and 
$$
\partial\,\Horn_n(\lambda,\mu) : = \Horn_n(\lambda,\mu)\setminus \Horn_n(\lambda,\mu)^\circ
$$
denote the  
\textit{boundary} of $\Horn_n(\lambda,\mu)$.
\end{definition}

In particular,  if $\Horn_n(\lambda,\mu)$ belongs to an affine subspace of dimension $n-2$,
then, by this definition, all points of $\Horn_n(\lambda,\mu)$ are automatically boundary points: $\partial\, \Horn_n(\lambda,\mu) = \Horn_n(\lambda,\mu)$.

\begin{remark}
The boundary $\partial\,\Horn_n(\lambda,\mu)$ is \textit{not} the usual boundary
%of the subset $\Horn_n(\lambda,\mu)\subset \R^n$ 
with respect to the Euclidean topology on $\R^n$.   For example, it may not include some points of $\Horn_n(\lambda,\mu)$ lying on facets of the cone $\mathfrak{C}_n$.
The boundary $\partial\, \Horn_n(\lambda,\mu)$ is rather the intersection of the Euclidean boundary of the symmetrization
$$
\mathrm{Sym}(\Horn_n(\lambda,\mu)):=\bigcup_{w\in S_n}w(\Horn_n(\lambda,\mu))
$$
with the cone $\mathfrak{C}_n$.
\end{remark}

\begin{definition}
Let us say that a Horn triple $(\lambda,\mu,\nu)$ is \textit{splittable} if there exist Horn triples 
$(\lambda',\mu',\nu')$ and $(\lambda'',\mu'',\nu'')$ such that $\lambda=\lambda'\oplus\lambda''$, $\mu=\mu'\oplus \mu''$, and 
$\nu=\nu'\oplus\nu''$.  
\end{definition}

Let us formulate the main result of this section.

\begin{theorem}\label{thm:boundary_splitting_horn}
Let $n\geq 2$.
    For any $\lambda, \mu\in \mathfrak{C}_n$ and any boundary point
    $\nu\in \partial\,\Horn_n(\lambda,\mu)$, the Horn triple $(\lambda,\mu,\nu)$ is splittable.
\end{theorem}

\begin{corollary}
If $\Horn_n(\lambda,\mu)$ belongs to an affine subspace of dimension $n-2$, then all
Horn triples $(\lambda,\mu,\nu)$ are splittable.
%, for $\nu\in \Horn_n(\lambda,\mu)$, is splittable.
\end{corollary}

Theorem~\ref{thm:boundary_splitting_horn} implies the splitting axiom using the following two simple observations.

\begin{lemma}  Let $\nu$ be an extreme point of $\Horn_n(\lambda,\mu)$ then either $\nu\in\partial\, Horn_n(\lambda,\mu)$ or
$\nu= c\,(1,\dots,1)$.
\end{lemma}

\begin{proof}
    Assume that $\nu\in\Horn_n(\lambda,\mu)^\circ$ is an interior point.  Then there exists $\epsilon>0$ such 
    that $B_\epsilon(\nu)\cap \mathfrak{C}_n \subset \Horn_n(\lambda,\mu)$.  If $\nu$ is not of the form $c\,(1,\dots,1)$,
    then $B_\epsilon(\nu)\cap \mathfrak{C}_n$ contains the line segment $[\nu-b,\nu+b]$, for some nonzero vector $b$.
    For example, we can take the same vector $b$ as the vector we used in the proof of Theorem~\ref{thm:prop:breathing_paths}
    (the splitting axiom for honeycombs).
    Thus $\nu$ is not an extreme point, which proves the needed claim.
\end{proof}

\begin{lemma}
Any Horn triple $(\lambda,\mu,\nu)$ with $\nu=c\,(1,\dots,1)$ is splittable.
\end{lemma}

\begin{proof}  
In this case, we have splitting for obvious reasons.   Indeed, $A+B$ should be $c\,I$, and thus matrices $A$, $B$, $A+B$
are simultaneously diagonalizable. 
\end{proof}

Let us first prove Theorem~\ref{thm:boundary_splitting_horn} in the case when $\mu$ and $\nu$ have distinct parts.

\subsection{The case of $\mu$ and $\nu$ with distinct parts}

%Let us say that a point $\nu\in\Horn_n(\lambda,\mu)$ is a \textit{boundary point} if $\dim(T_\nu(\Horn_n(\lambda,\mu)))\leq n-2$.
%In particular, if $\Horn_n(\lambda,\mu)$ belongs to an affine subspace of dimension $n-2$,
%then all points of $\Horn_n(\lambda,\mu)$ are automatically boundary points.

\begin{proposition}\label{prop:vertex_splitting}
    Let $n\geq 2$.   Let  $\nu\in\partial\,\Horn_n(\lambda, \mu)$ be a boundary point.  Assume that both $\mu$ and $\nu$ have all distinct parts.
    %, i.e., $\mu$ and $\nu$ are strictly in the interior of the fundamental chamber $\mathfrak{C}_n$. 
    Then any Hermitian matrices $A, B, A+B$ with spectra $\lambda, \mu, \nu$, respectively, are simultaneously block diagonalizable
    with two nontrivial diagonal blocks.
    %i.e., there exists a unitary matrix $U$ such that $U^{-1}AU$, $U^{-1}BU$, and thus $U^{-1}(A+B)U$ are block diagonal with two square blocks 
    %of sizes $m\geq 1$ and $(n-m)\geq 1$ each.
    In particular, the spectra of all blocks are given by subsequences of $\lambda, \mu$, $\nu$; 
    thus the Horn triple $(\lambda,\mu,\nu)$ is splittable.
%  We have such simultaneous block diagonalization, for boundary points $\nu\in\Horn_n(\lambda,\mu)$ if $\dim\Horn_n(\lambda,\mu)=n-1$;
%   and for all points $\nu\in\Horn_n(\lambda,\mu)$ if $\dim\Horn_n(\lambda,\mu) \leq n-2$.
\end{proposition}

%We will prove Proposition~\ref{prop:vertex_splitting} in the next section using the method outlined above.

%\begin{corollary}\label{cor:vertex_splitting_interior}
%Any triple $(\lambda,\mu,\nu)$ that satisfies the conditions of Proposition~\ref{prop:vertex_splitting}
%is splittable.
%\end{corollary}

\begin{remark}
    Our approach to proving Proposition~\ref{prop:vertex_splitting}
    relies on the multiplicity-freeness of $\nu$; otherwise, the derivatives of the coordinates $\nu_i$ are not well-defined. However, 
    %given the fact that $\Horn_n(\lambda, \mu)$ is not just a closed set but a polytope, 
    we may use this result to prove splitting for all extreme points $\nu\in \Horn_n(\lambda,\mu)$ with repeated parts.  For this extension, it is important to prove 
    Proposition~\ref{prop:vertex_splitting} not just for extreme points, but for all boundary points of
    $\Horn_n(\lambda,\mu)$ in the interior of the fundamental Weyl chamber $\mathfrak{C}_n$. 
\end{remark}

Recall  that we assumed that $A= U D_\lambda U^{-1}$, $B=D_\mu$, and $A+B=UD_\lambda U^{-1}+ D_\mu $ has spectrum $\nu$,
where $\mu$ and $\nu$ have all distinct parts.  
Given a skew-Hermitian matrix $S$ and $t\in \R$, let $\nu(t)= \nu_S(t)$ be the spectrum 
$$
\nu(t):= \mathrm{spec}((e^{tS}U)D_\lambda (e^{tS}U)^{-1} + D_\mu) =
 \mathrm{spec}(UD_\lambda U^{-1} + e^{-tS} D_\mu e^{tS}) =
 \mathrm{spec}(A + e^{-tS}Be^{tS}),
$$
which belongs to $\Horn_n(\lambda,\mu)$ by definition.
Then $\nu(t)$ is a continuous differentiable function of $t$ defined on some interval $[-a,a]$.
Here $a>0$ is taken to be sufficiently small so that
no pair of eigenvalues collide while $t$ ranges in the interval $[-a,a]$.
(This step only makes sense because we assumed that the parts of $\nu$ are distinct.)
Let $d_S\in\R^n$ be the derivative of $\nu_S(t)$ at $t=0$:
%\begin{equation}\label{eq:d_S}
$$
d_S := \left({d\,\nu_S(t)\over dt}\right)_{t=0}
$$
%\end{equation}
The map $S\mapsto d_S$ is an $\R$-linear map from the space (Lie algebra) of skew-Hermitian matrices $S$ to $\R^n$.

Define the \textit{tangent space} $T_\nu=T_\nu(\lambda,\mu)$ 
to $\Horn_n(\lambda,\mu)$ at point $\nu$ as the image of this map:
$$
T_\nu =T_\nu(\lambda,\mu) := \{d_S\mid S \textrm{ is a skew-Hermitian}\}\subset \R^n.
$$
Clearly, $T_\nu$ is a linear subspace of $\R^n$. Moreover, $T_\nu$ is a subspace of the 
hyperplane $H=\{x_1+\dots+x_n=0\}$.  Thus the dimension of $T_\nu$ is at most $n-1$. 

\begin{lemma}\label{lem:boundary=>dim}
Assume that $\nu$ has all distinct parts.
If $\nu\in\partial\,\Horn_n(\lambda,\mu)$ is a boundary point, then 
$\dim T_\nu(\lambda,\mu)\leq n-2$.
\end{lemma}

\begin{proof}
Suppose the tangent space $T_\nu(\lambda,\mu)$ has maximal possible dimension $n-1$. Then the map $S\mapsto\nu_S(1)$ from the space
of skew-Hermitian matrices $S$ to the affine hyperplane $\{x_1+\cdots+x_n = \sum \lambda_i + \sum \mu_i\}\subset\R^n$ is locally surjective
at $S=0$.   Indeed, the (matrix of the) linear map $S\mapsto d_S$ is the Jacobian of the map $S\mapsto \nu_S(1)$ at $S=0$.
Surjectivity of the Jacobian of a continuously differentiable map implies local surjectivity of the map.
Thus $\Horn_n(\lambda,\mu)$ contains a small $(n-1)$-dimensional ball $\mathrm{B}_\epsilon(\nu)$ at $\nu$ for some $\epsilon>0$.
Thus $\nu$ is not a boundary point of $\Horn_n(\lambda,\mu)$, 
which implies the needed claim.
\end{proof}

%Now assume that $\nu$ has all distinct parts; so the above discussion applies.
%If $\nu$ is an internal point of $\Horn_n(\lambda,\mu)$, see Definition~\ref{def:boundary_points}, then $\dim T_\nu(\Horn_n(\lambda,\mu))=n-1$. 
%To prove the lemma in the opposite direction, assume that $\dim T_\nu(\Horn_n(\lambda,\mu))=n-1$, that is, $T_\nu(\Horn_n(\lambda,\mu))
%=\mathrm{image}(S\to d_S)$ is the entire hyperplane $\{x_1+\cdots+x_n=0\}\subset\R^n$.
%The Jacobian of the map $S\to\nu(S)$ at $S=0$, which is exactly the matrix of the linear map $S\to d_S$, has maximal possible rank $n-1$.
%Thus the map $S\to\nu(S)$ is locally surjective at $S=0$.   This implies that the image of the map $S\to\nu(S)$
%contains a small $\epsilon$-neighborhood $B_\epsilon(\nu)$ of $\nu$.  Thus $\nu$ is not a boundary point.
%\end{proof}

Let us now calculate the derivatives $d_S$.   We will apply the following known fact, whose proof is left as an exercise for the reader.

\begin{lemma}\label{lem:commutatorDS}   
The commutator $[D, S] := D S - S D$ of a real diagonal matrix $D$ and a skew-Hermitian matrix $S$ is a 
skew-Hermitian matrix whose diagonal entries are always zero and off-diagonal entries are given by
$$
[D,S]_{ij} = (d_{ii} - d_{jj}) s_{ij}.
$$
\end{lemma}

Let $Y$ be a unique (up to rescaling of columns) unitary matrix that diagonalizes $A+B$:
%= UD_\lambda U^{-1} + D_\mu$:
$$
A+B = Y D_\nu Y^{-1}
$$
The columns vectors $Y_i$ of $Y$ are the eigenvectors of $A+B$ with eigenvalues $\nu_i$.

Let $\langle u,v\rangle:= u^* v$ denote the standard Hermitian inner product of two complex column vectors.
%$u$ and $v$ .  

\begin{lemma}\label{lem:d_S} The vector $d_S$ is the vector formed by the diagonal entries of the matrix
$Y^*[B,S]\,Y$, where $[B,S]:=BS-SB$, the commutator of matrices $B$ and $S$.   
Equivalently, the $i$th coordinate of vector $d_S$ is 
$$
Y_i^*[B,S]\,Y_i = \langle Y_i, [B,S]\,Y_i\rangle.
$$
\end{lemma}

\begin{proof} 
Fix a skew-Hermitian matrix $S$.  Let $B(t) := e^{-tS}B e^{tS}$.
Let $Y(t)$ be a family of unitary matrices that diagonalize $A+ B(t)$,
i.e.,
$A+B(t) = Y(t) D_{\nu(t)} Y^{-1}(t)$, with $Y(0)=Y$. We have
$$
D_{\nu(t)} = Y(t)^*(A + B(t))Y(t).
$$
By differentiating the product of matrices in the right hand side using the chain rule, we obtain 
$$
D_{\nu'(t)}=
Y'(t)^* (A + B(t))Y(t) + Y(t)^* B'(t)Y(t) + Y(t)^*(A + B(t))Y'(t),
$$
where derivatives of matrices are taken termwise.
We have 
$$
B'(t) = -S e^{-tS}B e^{tS} + e^{-tS}B S e^{tS}.
$$
Specializing at $t=0$, we get
$$
D_{\nu'(0)}= Y'(0)^*(A+B)Y + Y^*[B,S]Y + Y^*(A+B)Y'(0). 
$$

We claim that the sum of two terms $Z:=Y'(0)^*(A+B)Y+Y^*(A+B)Y'(0)$ in the right hand side of the above equation has all diagonal entries zero.
Indeed, since $Y(t)$ is unitary,
taking the derivative of $Y(t)^*\, Y(t)=I$ at $t=0$, we get
$$
Y'(0)^* \,Y + Y^* \,Y'(0) = 0. 
$$
Thus $Y'(0)^*\, Y = -  Y^*\, Y'(0)$ is skew-Hermitian.
Since $(A+B)Y = Y D_\nu$ and $Y^*(A+B) = D_\nu Y^*$, we get
$$
Z= Y'(0)^* \,Y \, D_\nu + D_\nu \,Y^* \,Y'(0) = [Y'(0)^*\,Y, D_\nu],
$$
and the commutator of a diagonal and a skew-Hermitian matrices results in a skew-Hermitian matrix whose diagonal entries are zero,
see Lemma~\ref{lem:commutatorDS}.

Finally, $\nu'(0)$ is exactly the vector $d_S$.   Thus $d_S$ is the vector formed by the diagonal entries of matrix $Y^*[B,S]Y$, as needed. 
\end{proof}

The space of skew-Hermitian matrices $S$ is the $\R$-span of the 
matrices $S_{pq}^\tau := \tau E_{pq} - \bar\tau E_{qp}$, where $p,q\in[n]$,
$E_{pq}$ is the matrix with $1$ in position $(p,q)$ and zeroes elsewhere, and  $\tau\in\C$.
%and $E_{pq}$ is the matrix with $1$ in position $(p,q)$ and zeroes elsewhere. 
Let
$$
d_{pq}^\tau := d_{S_{pq}^\tau}\in\R^n
$$
Let us specialize Lemma~\ref{lem:d_S} to $S=S_{pq}^\tau$.

%\begin{definition}\label{def:V_{nu}}
%Consider the derivative
%$$
%d_{pq}^\theta := \left(d\,\nu^{\theta}_{pq}(t)/dt\right)_{t=0}\in\R^n
%$$
%of function $\nu_{pq}^\theta(t)$ at $t = 0$.  
%Clearly, vector $d_{pq}^\theta$ belongs to the tangent space $T_\nu(\Horn_n(\lambda,\mu))$.

    %as $T$ shifts infinitesimally in the direction of $T^{\theta}_{(p,q)}$. 
%    Let $V_{\nu}\subset\R^n$ be the $\R$-span of all vectors of the form $d_{pq}^\theta$:
%    \[V_{\nu} := \Span_{\R}(d^{\theta}_{pq})\]
%\end{definition}

%Clearly, $V_\nu$ is a subspace of the tangent space $T_\nu(\Horn_n(\lambda,\mu))$.
%Notice that $V_{\nu}$ needs to be a \textit{real} vector space: it represents all directions in which $\nu$ can vary, and spectra of Hermitian matrices are always real.

\begin{lemma}\label{lemma:computation}
The $i$th coordinate of vector $d^{\tau}_{pq}$ is given by 
\[2(\mu_p-\mu_q)\mathrm{Re}(\tau \overline{Y}_{pi}Y_{qi} )\]
\end{lemma}

\begin{proof}
Calculating the commutator $[B,S]$ for $B = D_\mu$ and $S=S_{pq}^\tau$ using Lemma~\ref{lem:commutatorDS}, we obtain that
%Finally, specializing \eqref{eq:nu_diff_commutator}
%at $S=S_{pq}^\tau := \tau E_{pq} -\bar\tau E_{qp}$,
%recalling that $B$ is the diagonal matrix $D_\mu$, and calculating the commutator $[B,S]$, we obtain that
the $i$th coordinate of $d_{pq}^\tau$ is
$$
%\langle(\mu_p-\mu_q)(\tau E_{pq}+\bar\tau E_{qp})Y_i,Y_i\rangle =
(\mu_p-\mu_q) Y_i^*(\tau E_{pq}+\bar\tau E_{qp})Y_i = 
(\mu_p-\mu_q)\mathrm{Re}(\tau \overline{Y}_{pi}Y_{qi}),
$$
as needed.
\end{proof}

%\begin{proposition}\label{prop:dim}
%    If $\nu$ is a boundary point of $\Horn_n(\lambda, \mu)$, the dimension of the real vector space $V_{\nu}$ is at most $n-2$. 
%\end{proposition}

%\begin{proof}
%   The restriction $\sum_{i} \nu_i = \sum_i \lambda_i + \sum_i \mu_i$ means that $\Horn_n(\lambda, \mu)$ is contained in a codimension-$1$ hyperplane of $\R^n$. The space $V_{\nu}$ is spanned by tangent vectors to differentiable curves passing through $\nu$ that stay within $\Horn_n(\lambda, \mu)$. If these tangent vectors span the entire hyperplane, then $\nu$ could not have been a boundary point of $\Horn_n(\lambda, \mu)$. 
%\end{proof}

We will need the following general fact about complex vector spaces.

\begin{lemma}\label{lemma:Re_complex}
    For any complex vector space $V\subset\C^n$ and $\mathrm{Re}(V):=\{\mathrm{Re}(v)\mid v\in V\}\subset\R^n$, we have
    \[\dim_{\mathbb{C}}(V) \leq \dim_{\R}(\mathrm{Re}(V))\]
\end{lemma}

\begin{proof} Let $\mathrm{Re}(V)\otimes \C\subset\C^n$ be the $\C$-span of vectors in $\mathrm{Re}(V)$.  We have
$V\subset \mathrm{Re}(V)\otimes \C$.  Indeed, for $v\in V$, we have $\mathrm{Re}(v), \mathrm{Re}(i\,v)\in\mathrm{Re}(V)$
and $v=\mathrm{Re}(v) - i\, \mathrm{Re}(i\,v)\in \mathrm{Re}(V)\otimes \C$.
Thus 
$
\dim_\C V \leq \dim_\C (\mathrm{Re}(V)\otimes\C) = 
\dim_\R \mathrm{Re}(V)$.
\end{proof}

%\begin{definition}\label{def:G}

    Let $G = (G_{i,pq})$ be the $n \times n(n-1)$ \textit{complex} matrix whose $(i, pq)$th entry is $G_{i,\,pq}=\overline{Y_{pi}}Y_{qi}$,
    for all $i,p,q\in[n]$ such that $p\ne q$.
%\end{definition}

\begin{lemma}\label{lemma:n-2}
If $\nu\in \partial \, \Horn_n(\lambda,\mu)$ is a boundary point,
    then the matrix $G$ has \textit{complex} rank at most $n-2$. 
\end{lemma}

\begin{proof}
Let $V\subset\C^n$ be the complex vector space spanned by the columns of matrix $G$, that is,
by the vectors $G_{pq}$ whose $i$th coordinate is $\overline{Y_{pi}}Y_{qi}$. 
Then $\mathrm{Re}(V)\subset\R^n$ is the real vector space spanned by the vectors $\mathrm{Re}(\tau\, G_{pq})$,
for all pairs $(p,q)$, $p\ne q$, and $\tau \in \C$, which are exactly the tangent vectors $d_{pq}^\tau \in T_\nu (\lambda,\mu)$
rescaled by the factors $1\over{2(\mu_p-\mu_q)}$, see Lemma~\ref{lemma:computation}.
%Since $\nu$ is a boundary point of $\Horn_n(\lambda,\mu)$, 
(For this step, we need the assumption that parts of $\mu$ are distinct.)
Using  the bound $\dim T_\nu(\lambda,\mu)\leq n-2$ provided by Lemma~\ref{lem:boundary=>dim}, 
and Lemma~\ref{lemma:Re_complex}, we obtain
$$
\mathrm{rank}_\C(G)=\dim_\C V \leq \dim_\R\mathrm{Re}(V) = \dim_\R T_\nu\leq n -2,
$$
as needed.
\end{proof}

Define a related real $n \times n$ matrix $F$, whose $(i,p)$th entry is $F_{ip} = \overline{Y_{pi}} Y_{pi}= \lvert Y_{pi}\rvert^2$.
Since the matrix $Y$ is unitary, the matrix $F$ is doubly stochastic by construction.

\begin{lemma}\label{lemma:decomp2}
    The matrix $FF^*$ is decomposable. That is, there exists a permutation matrix $P$ and square matrices $M_1$ and $M_2$ such that 
    \[FF^* = P \begin{pmatrix}
        M_1 & 0\\
        0 & M_2
    \end{pmatrix}P^{-1}\]
\end{lemma}

\begin{proof}
We have
$$
FF^*+GG^* = I.
$$
Indeed,  the $(i,j)$th entry of $FF^*+GG^*$ is
    $$
    (FF^*+GG^*)_{ij} = \sum_{p=q} \bar Y_{pi} Y_{qi} Y_{jp} \bar Y_{jq} + \sum_{p\neq q} \bar Y_{pi} Y_{qi} Y_{pj} \bar Y_{qj}
    = \left(\sum_{p}\bar Y_{pi}Y_{jp}\right)\left(\sum_q Y_{qi}\bar Y_{jq}\right) = \delta_{ij},
    $$
    because $Y$ is unitary.
Since $GG^*$ has rank $\leq n-2$, that is, it has multiple eigenvalue 0, the matrix $FF^* = I - GG^*$ has multiple 
eigenvalue 1.   The matrix $FF^*$ is doubly stochastic because it is the product of two doubly stochastic matrices $F$ and $F^*$.
Now decomposability of $FF^*$ follows from the key Lemma~\ref{lemma:perron-frobenius}.
%, it suffices to show that $FF^*$ is stochastic and has $1$ as a multiple eigenvalue. 
    %Because $\{Y_i\}_{i \leq n}$ was an \textit{orthonormal} eigenbasis, $F$ is doubly stochastic,
    %The matrix $FF^*$ is doubly stochastic. Indeed, the product of doubly stochastic matrices $F$ and $F^*$ should be
    %doubly stochastic. 
    %
    %Next, we can compute 
    %\[GG^* = I - FF^*\]
    %using the orthonormality of the basis formed by columns $Y_i$ of matrix $Y$. Since the kernel of $GG^*$ has dimension at least $2$ by Lemma %\ref{lemma:n-2}, $1$ is indeed a multiple eigenvalue of $FF^*$. 
\end{proof}

\begin{lemma}\label{lem:F_decomposable}
    There exist permutation matrices $P, R$ and square matrices $N_1, N_2$ such that 
    \[F = P\begin{pmatrix}
        N_1 & 0\\
        0 & N_2
    \end{pmatrix}R\]
\end{lemma}

\begin{proof}
    Let $F = P\begin{pmatrix}
    L_{11} & L_{12}\\
    L_{21} & L_{22}
\end{pmatrix}P^{-1}$, where $P$ is the permutation matrix from the decomposition of $FF^*$ given by Lemma \ref{lemma:decomp2}
and matrices $L_{ij}$ are blocks of the corresponding sizes. Note that $P^* = P^{-1}$.
Then the matrix $P^{-1}FF^*P$ equals
\[
%PFF^*P^{-1}=
\begin{pmatrix}
    L_{11} & L_{12}\\
    L_{21} & L_{22}
\end{pmatrix}\begin{pmatrix}
    L_{11}^* & L_{21}^*\\
    L_{12}^* & L_{22}^*
\end{pmatrix} = \begin{pmatrix}
    M_1 & 0 \\
    0 & M_2
\end{pmatrix}\]
For the product on the left side to be block diagonal, we must have $L_{11}L_{21}^* + L_{12}L_{22}^*$ equals the zero matrix $0$.
%, and, similarly, $ZX^*+TY^*= 0$. 
Because all entries of $F$ are \textit{nonnegative
real} numbers by definition, it follows that $L_{11}L_{21}^* = L_{12}L_{22}^* = 0$. That is, whenever a column of $L_{11}$ has a nonzero entry, all entries in the corresponding column of $L_{21}$ are zero, and, similarly, for $L_{12}$ and $L_{22}$. Moving all nonzero columns of $L_{12}$ leftwards via a permutation matrix, and similarly for $L_{12},L_{21},L_{22}$, we find 
\[F = P\begin{pmatrix}
    N_1 & 0\\
    0 & N_2
\end{pmatrix}R\]
for some permutation matrix $R$.
The claim that $N_1, N_2$ are \textit{square} matrices follows from the fact that $F$ was doubly stochastic. 
\end{proof}

Since $F$ is the matrix whose entries
are squares of absolute values of entries of the unitary matrix $Y$, 
Lemma~\ref{lem:F_decomposable} implies the following claim.

\begin{lemma}\label{lem:prop_decomposability}
    There exist permutation matrices $P, R$ and nontrivial unitary matrices $W_1, W_2$ such that 
    \[Y = P \begin{pmatrix}
        W_1 & 0\\
        0 & W_2\\
    \end{pmatrix}R\]
\end{lemma}

Finally, let us show that Lemma~\ref{lem:prop_decomposability} implies Proposition \ref{prop:vertex_splitting}.

\begin{proof}[Proof of Proposition~\ref{prop:vertex_splitting}]
By the definition of $Y$, we have
%\begin{equation}\label{eq:Y}
$$
    UD_{\lambda}U^{-1} + D_{\mu} = Y D_{\nu} Y^{-1}.
$$
Plugging in $Y = P \begin{pmatrix}
        W_1 & 0\\
        0 & W_2\\
    \end{pmatrix}R$, 
    we left-multiply both sides of this equation
    by $P^{-1}$ and right-multiply by $P$, obtaining
%\begin{equation}\label{eq:W1}
$$
    P^{-1}UD_{\lambda}U^{-1}P + P^{-1}D_{\mu}P = 
    \begin{pmatrix}
    W_1 & 0 \\
    0 & W_2
    \end{pmatrix} R \,D_{\nu} R^{-1} 
    \begin{pmatrix}
        W_1^{-1} & 0 \\
        0 & W_2^{-1}
    \end{pmatrix}
$$
Since $P$ is a permutation matrix, $P^{-1}D_{\mu}P$ is also diagonal, with diagonal entries given by a permutation of $\lambda$. Similarly, the right hand side 
of the above equation is block diagonal. Thus first matrix in the left hand side should also be block diagonal.
We obtain that the three matrices 
 $A=U\,D_\lambda^{-1}$, $B=D_\mu $,
and $A+B=Y\,D_\nu Y^{-1}$
are simultaneously block-diagonalizable, and the eigenvalues in each block are given by subsequences of $\lambda, \mu$, and $\nu$. 
These subsequences give us the splitting. 
\end{proof}

\subsection{The case of repeated eigenvalues}

Let us now explain why Proposition~\ref{prop:vertex_splitting}, where we assume that $\mu$ and $\nu$ have distinct parts,
implies the general case of boundary point splitting
for arbitrary partitions $\lambda,\mu,\nu$, possibly with repeated parts.

%\begin{definition}
%Let us say that a Horn triple $(\lambda,\mu,\nu)$ is \textit{splittable} if there exist Horn triples 
%$(\lambda',\mu',\nu')$ and $(\lambda'',\mu'',\nu'')$ such that $\lambda=\lambda'\oplus\lambda''$, $\mu=\mu'\oplus \mu''$, and 
%$\nu=\nu'\oplus\nu''$.
%\end{definition}

%\begin{corollary}\label{cor:vertex_splitting_interior}
%Any triple $(\lambda,\mu,\nu)$ that satisfies the conditions of Proposition~\ref{prop:vertex_splitting}
%is splittable.
%\end{corollary}

\begin{lemma}\label{lem:Horn_splittable_closed}
    The set of all Horn triples $(\lambda,\mu,\nu)$ and its subset of splittable Horn triples are closed subsets of $\R^{3n}$ 
    in the usual topology on $\R^{3n}$ induced by the Euclidean norm.
\end{lemma}

\begin{proof}  
%{eq:horn=spec}
%By the definition,
%$$
%\Horn_n(\lambda,\mu)=\{\mathrm{spec}(UD_\lambda U^{-1} + D_\mu)\mid U \textrm{ is unitary}\}.
%$$
The facts that $\Horn_n(\lambda,\mu)$ is a compact subset of $\R^n$ that continuously depends on $\lambda$ and $\mu$,
and the set of all Horn triples $(\lambda,\mu,\nu)$ is a closed subset of $\R^{3n}$ follow immediately from the 
definition, see   \eqref{eq:horn=spec}.

For a subset $I=\{i_1,\dots,i_m\}\subset [n]$, let $\lambda_I:=(\lambda_{i_1},\dots,\lambda_{i_m})$, and let $I^c:=[n]\setminus I$.
%If $\lambda=(\lambda_1 \dots \lambda_n)\in \mathfrak{C}_n$ and $I = \{i_1, i_2 \dots i_m\} \subseteq [n]$, we use $\lambda_I$ to denote the partition $(\lambda_{i_1}, \lambda_{i_2} \dots \lambda_{i_m})$. The complement of $I$ in $[n]$ is denoted $I^c$, and $\lambda_{I^c}$ is defined analogously. For example, if $\lambda = (5,3,2,2)$, $I = \{1,3\}$, then $\lambda_I = (5,2)$ and $\lambda_{I^c} = (3,2)$. 
Let us say that a Horn triple $(\lambda,\mu,\nu)$ is \textit{$(I,J,K)$-splittable},
for proper non-empty subsets $I,J,K \subset [n]$ of the same cardinality $m$, if 
$\nu_K \in \Horn_{m}(\lambda_I, \mu_J)$ and $\nu_{K^c} \in \Horn_{n-m}(\lambda_{I^c}, \mu_{J^c})$. 
Clearly, a Horn triple $(\lambda,\mu,\nu)$ is splittable if and only if it is $(I,J,K)$-splittable for some $I$, $J$, $K$.
The set of $(I,J,K)$-splittable triples $(\lambda,\mu,\nu)\in\R^{3n}$ is a closed subset of $\R^{3n}$.
Thus the set of all splittable triples is also closed,
%The subset of points $\nu\in\Horn_n(\lambda,\nu)$ such that $(\lambda,\mu,\nu)$ is $(I,J,K)$-splittable is a compact subset of $\Horn_n(\lambda,\mu)$.
%(In fact, this subset is a convex polytope.) 
%Thus the subset of points $\nu\in\Horn_n(\lambda,\mu)$ such that $(\lambda,\mu,\nu)$ is splittable is also compact, 
because it is the union of a \textit{finite} collection of closed sets.
%Moreover, the set of all splittable Horn triples $(\lambda,\mu,\nu)$ is a closed subset of $\R^{3n}$.
\end{proof}

Now, we can finish the proof of Theorem~\ref{thm:boundary_splitting_horn}.

%\begin{proposition}\label{prop:vertex_splitting_general}
%For $n\geq 2$, let $\nu$ be any extreme point of $\Horn_n(\lambda,\mu)$.   Then the Horn triple $(\lambda,\mu,\nu)$ is splittable.
%\end{proposition}

\begin{proof}[Proof of Theorem~\ref{thm:boundary_splitting_horn}]
We need to show that any Horn triple $(\lambda,\mu,\nu)$, where $\nu=\mathrm{spec}(UD_\lambda U^{-1} + D_\mu)$ 
is a boundary point $\nu\in \partial\, \Horn_n(\lambda,\mu)$, is splittable.
%If $\nu=(a,\dots,a)$, then we have splitting for obvious reasons.   In this case $A+B$ should be $aI$, and thus matrices $A$, $B$, $A+B$
%are simultaneously diagonalizable.   In the rest of the proof, we assume that $\nu$ is not of this form.
%
By contradiction, assume that this Horn triple $(\lambda,\mu,\nu)$ is \textit{not} splittable.

%We will deduce that $\nu$ is \textit{not} an extreme point of $\Horn_n(\lambda,\mu)$.

Since the subset of splittable triples 
is closed and its complement is open, there exists $\epsilon>0$ such that
any triple $(\lambda, \mu',\nu')$
satisfying the conditions:
\begin{enumerate}
\item[(a)]
$\mu',\nu'\in\mathfrak{C}_n$,
\item[(b)] 
$\sum \lambda_i + \sum \mu'_i = \sum \nu'_i$,
\item[(c)]
$\lVert\nu'-\nu\rVert \leq \epsilon$ and $\lVert\mu'-\mu\rVert \leq \epsilon$,
\item[(d)] 
$\lVert\nu-\mathrm{spec}(U D_\lambda U^{-1} - D_{\mu'})\rVert\leq \epsilon$
\end{enumerate}
is also \textit{not} splittable. 

%A priori, we do not know that all such triples $(\lambda,\mu',\nu')$ are Horn triples.  
Let us show that all triples $(\lambda,\mu',\nu')$ satisfying conditions (a)--(d) are, in fact, Horn triples.
Assume $(\lambda,\mu',\nu')$ is a triple satisfying these conditions, which is not a Horn triple, and
such that $\mu'$ and $\nu'$ are 
\textit{strictly} in the interior $\mathfrak{C}_n$.   Consider the triple $(\lambda,\mu', \nu'')$,
where $\nu'' = \mathrm{spec}(D_\lambda - U^{-1} D_{\mu'} U)$.
By the definition, $(\lambda,\mu', \nu'')$ is a Horn triple and it also satisfies conditions (a)--(d).

Let us connect these two triples by the straight line path $p:[0,1]\to\ \R^{3n}$ such that $p(0)=(\lambda,\mu',\nu')$ and $p(1)=
(\lambda,\mu',\nu'')$.
Define the number $t_{\mathrm{inf}} \in[0,1]$ as the infimum:
$$
t_{\mathrm{inf}} = \mathrm{inf} \{t\in[0,1] \mid p(t) \textrm{ is a Horn triple}\} 
$$
Then
the triple $(\lambda,\mu', \nu'''):=p(t_{\mathrm{inf}})$ is a Horn triple, because, by Lemma~\ref{lem:Horn_splittable_closed}, the set of Horn triples is closed.  
The triple $(\lambda,\mu',\nu''')$ also satisfies the above conditions (a)--(d).
We have $t_{\mathrm{inf}}>0$, because we assumed that $(\lambda',\mu',\nu')=p(0)$ is not a Horn triple.

Now, by our definition of boundary points, $\nu'''$ is a boundary point of $\Horn_n(\lambda,\mu')$, because any 
$\delta$-neighborhood $\mathrm{B}_\delta(\nu''')$, $\delta>0$, of this point contains a point with non-Horn triple $p(t)$, where $t$ is slightly smaller that $t_{\mathrm{inf}}$.
This means that all conditions of Proposition~\ref{prop:vertex_splitting} hold for the triple $(\lambda,\mu',\nu''')$.
By Proposition~\ref{prop:vertex_splitting}, the triple $(\lambda,\mu',\nu''')$ is splittable.
We get a contradiction.

This shows that all triples $(\lambda,\mu',\nu')$ satisfying conditions (a)--(d), with $\mu'$ and $\nu'$ strictly in the interior of $\mathfrak{C}_n$ are Horn triples.
Since, by Lemma~\ref{lem:Horn_splittable_closed}, the set of Horn triples is closed, we deduce that all triples $(\lambda,\mu',\nu')$ satisfying (a)--(d),
including the triples with $\mu'$ and $\nu'$ on the boundary of $\mathfrak{C}_n$, are Horn triples.

Let us now specialize the conditions (a)--(d) to the case when $\mu'=\mu$.   We proved that $\Horn_n(\lambda,\mu)$ contains all
points $\nu'$ such that (a) $\nu'\in\mathfrak{C}_n$; (b) $\nu'$ belongs to the affine hyperplane 
$\{x_1+\cdots+x_n = \sum \lambda_i + \sum \mu_i\}$; and (c) $\rVert\nu'-\nu\rVert \leq \epsilon$.
This set is exactly the intersection of the $(n-1)$-dimensional ball $\mathrm{B}_\epsilon(\nu)$ with $\mathfrak{C}_n$.
Thus $\nu$ is not a boundary point.  We got a contradiction with our assumption that $\nu$ is a boundary point, which proves theorem in the general case.
%Finally, for all points $\nu\in\mathfrak{C}_n$, except points of the form $\nu=(a,\dots,a)$, the set
%$B_\epsilon(\nu) \cap \mathfrak{C}_n$ contains a line segment of the form $[\nu-b, \nu+b]$, for some nonzero vector $b$.
%%For example, we can take some vector $b$ of the form $\epsilon(e_i-e_j)/\sqrt{2}$, which goes along an edge of the cone $\mathfrak{C}_n$.
%More precisely, we can take the same vector $b$ as the vector we used in the proof of Theorem~\ref{thm:prop:breathing_paths}
%(vertex splitting for $\Honey_n(\lambda,\mu)$).
%Thus $\nu$ is not an extreme point of $\Horn_n(\lambda,\mu)$.   
%
%We proved that all non-splittable points of $\Horn_n(\lambda,\mu)$
%are not extreme points, as needed.
\end{proof}

\bibliographystyle{alpha}
\bibliography{refs}

\end{document}